\documentclass[11pt, leqno]{article}
\usepackage{amsmath,amssymb,amsthm, epsfig,bm}
\usepackage{comment}
\usepackage[mathscr]{euscript}
\usepackage{hyperref}
\usepackage[pdftex]{color}
\usepackage{mathtools}
\usepackage{color}
\usepackage{ulem}
\usepackage{mathrsfs}
\usepackage{wasysym}
\usepackage{stackrel}
\usepackage{pgfplots}
\pgfplotsset{compat=newest}
\allowdisplaybreaks

\title{\bf \Large Gradient regularity for viscosity solutions to quasilinear parabolic equations with mixed singular-degenerate structure}

\author{\it by \smallskip \\ Junior da Silva Bessa \footnote{\noindent Universidade Estadual de Campinas - UNICAMP. Instituto de Matem\'{a}tica, Estat\'{i}stica e Computa\c{c}\~{a}o Cient\'{i}fica - IMECC. Departamento  de Matemática. Bar\~{a}o Geraldo, Campinas - SP, Brazil. \noindent \texttt{E-mail address: \url{jbessa@unicamp.br}}},\qquad
Jo\~{a}o Vitor  da Silva
\footnote{\noindent Universidade Estadual de Campinas - UNICAMP Instituto de Matem\'{a}tica, Estat\'{i}stica e Computa\c{c}\~{a}o Cient\'{i}fica - IMECC. Departamento  de Matemática. Bar\~{a}o Geraldo, Campinas - SP, Brazil. \noindent \texttt{E-mail address: \url{jdasilva@unicamp.br}}},\\ \quad $\&$ \\ \quad Ginaldo de Santana S\'{a}\footnote{\noindent Centro de Modelamiento Matemático (CNRS IRL2807), Universidad de Chile, Santiago, Chile. \noindent \texttt{E-mail address: \url{gdesantana@dim.uchile.cl}}}
}

\newlength{\hchng}
\newlength{\vchng}
\newcommand{\defeq}{\mathrel{\mathop:}=}

\newtheorem{theorem}{Theorem}[section]

\newtheorem{lemma}[theorem]{Lemma}

\newtheorem{proposition}[theorem]{Proposition}

\theoremstyle{definition}

\newtheorem{definition}[theorem]{Definition}

\theoremstyle{remark}

\numberwithin{equation}{section}

\date{}

\begin{document}

\maketitle
\begin{abstract}
\noindent We establish regularity results for viscosity solutions to a class of quasilinear parabolic equations exhibiting nonhomogeneous degeneracy or singularity (a double phase regime) of the form
\[
u_t - \big(|Du|^{\mathfrak{p}} + \mathfrak{a}(x,t)|Du|^{\mathfrak{q}}\big)\Delta_p^{\mathrm N} u = f(x,t) \quad \text{in } Q_1,
\]
where $-1 < \mathfrak{p} < 0$, $\mathfrak{p} \leq \mathfrak{q}$, and $\mathfrak{a}, f : Q_1 \to \mathbb{R}$ are prescribed functions. Using the Jensen--Ishii method, we prove Lipschitz regularity for appropriately translated solutions. Moreover, combining this approach with intrinsic scaling techniques, we establish interior Hölder continuity estimates for the gradient. Our results extend recent work of Fang and Zhang on the homogeneous case via a different approach. 

\vspace{0.2cm}
	
\end{abstract}
\section{Introduction}

In this manuscript, we study the regularity of viscosity solutions to the quasilinear parabolic problem
\begin{equation}\label{Problem}
 u_{t}(x,t) - \mathcal{H}(x,t,Du)\,\Delta_{p}^{\mathrm{N}}u(x,t) = f(x,t)\quad \text{in } Q_1 := B_1 \times (-1,0],
\end{equation}
where $1 < p < \infty$ and $\Delta_{p}^{\mathrm{N}}$ denotes the \textit{normalized $p$-Laplacian}, given by
\[
\Delta_{p}^{\mathrm{N}}u(x,t) = \Delta u + (p-2)\left\langle D^2 u \frac{Du}{|Du|}, \frac{Du}{|Du|} \right\rangle.
\]
The diffusion coefficient $\mathcal{H}:\mathbb{R}^n \times \mathbb{R} \times \mathbb{R}^n \to \mathbb{R}$ is defined as
\[
\mathcal{H}(x,t,\xi) = |\xi|^{\mathfrak{p}} + \mathfrak{a}(x,t)|\xi|^{\mathfrak{q}},
\]
where $\mathfrak{a}$ is a positive modulating function and the exponents satisfy $-1 < \mathfrak{p} \leq \mathfrak{q} < \infty$ with $\mathfrak{p} < 0$.

Beyond their intrinsic mathematical interest, problems of the form \eqref{Problem} arise in connection with stochastic tug-of-war games (cf. \cite{Han2022,MPR10,Parv2024,ParRuost16}) and in models for image processing (cf. \cite{Does11}).

Substantial progress has recently been made in the regularity theory of quasilinear double phase problems related to \eqref{Problem}. In the variational setting, we recall the work of De Filippis \cite{DeF20}, who studied weak solutions to the Cauchy--Dirichlet problem
\[
\left\{
\begin{array}{rcll}
\partial_t u - \mathrm{div}\!\left( (|Du|^{p-2} + a(x,t)|Du|^{q-2})Du \right) &=& 0 & \text{in } \Omega_T,\\
u &=& g & \text{on } \partial_{\mathrm{par}}\Omega_T,
\end{array}
\right.
\]
for $g \in C_{\mathrm{loc}}(\mathbb{R}; L^2_{\mathrm{loc}}(\mathbb{R}^n)) \cap L^r_{\mathrm{loc}}(\mathbb{R}; W^{1,r}_{\mathrm{loc}}(\mathbb{R}^n))$, where $r = p'(q-1)$. In this framework, quantitative gradient estimates were established, along with extensions to $p(x,t)$-type diffusion and anisotropic settings. For the same equation without Dirichlet boundary conditions, Buryachenko and Skrypnik \cite{BurSkr22} proved continuity of weak solutions and a Harnack inequality.

Recently,  Sen and  Siltakoski in \cite{SS2025} proved local spatial Lipschitz regularity for viscosity solutions to the parabolic double phase equation
\[
\partial_t u - \operatorname{div}\left(|Du|^{p-2}Du + a(z)|Du|^{q-2}Du\right) = f(z,Du),
\]
via the Ishii--Lions method. They further establish Hölder continuity in time, which is optimal in the degenerate regime. The coefficient \(a\) is nonnegative, bounded, locally Lipschitz continuous in space, and continuous in time, while the source term \(f\) is assumed continuous and satisfies a suitable gradient growth condition. In addition, they proved the equivalence between bounded viscosity and weak solutions under mild additional regularity assumptions on \(a\).

In the non-divergence setting, Fang and Zhang \cite{FZ23} investigated
\begin{equation}\label{EqFZ}
 u_{t}(x,t) = \big(|Du|^{q} + \mathfrak{a}(x,t)|Du|^{s}\big)\Delta_{p}^{\mathrm{N}}u(x,t)\quad \text{in } Q_1,
\end{equation}
with $-1 < q \leq s < \infty$. This model exhibits nonhomogeneous degeneracy and/or singularity, with $\mathfrak{a} \in C^{1}(Q_1)\cap L^{\infty}(Q_1)$ and $\mathfrak{a} > 0$. The authors proved Hölder continuity of the gradient via an oscillation-decay scheme along shrinking parabolic cylinders (see also \cite{IJS19} for the one-phase case).

We also mention \cite{BessaDaSilvaSa25}, where the regime $0 \leq \mathfrak{p} \leq \mathfrak{q}$ (doubly degenerate structure) was considered. By combining the method of alternatives with an improvement-of-flatness argument, the authors established Hölder continuity of the gradient for viscosity solutions to \eqref{Problem}, assuming that $f$ is bounded and $\mathfrak{a}$ is positive, bounded, and of class $C^{1}$. They also obtained non-degeneracy results and optimal growth estimates near critical points.

For further regularity results concerning normalized $p$-Laplacian models with degeneracy and/or singularity, see \cite{Attouchi20,AtouParv18,AttouRuos20,AttouRuos18,Han2020}.

Despite these advances, the analysis of \eqref{Problem} in the regime $\mathfrak{p} < 0$ remains largely open. The main difficulty stems from the mixed gradient growth, reflecting the interplay between singular and degenerate terms, namely $|Du|^{\mathfrak{p}}$ and $\mathfrak{a}(x,t)|Du|^{\mathfrak{q}}$. To the best of our knowledge, available results are essentially confined to the homogeneous case \eqref{EqFZ}.

Motivated by this gap, we establish Hölder continuity of the gradient for viscosity solutions to \eqref{Problem} within a mixed-growth framework.

We impose the following structural assumptions on \eqref{Problem}:
\begin{itemize}
\item[{\bf(A1)}]{\textbf{(Growth condition)}} The exponents $\mathfrak{p}$ and $\mathfrak{q}$ satisfy
\[
-1 < \mathfrak{p} < 0 \quad \text{and} \quad \mathfrak{p} \leq \mathfrak{q} < \infty.
\]
\item[{\bf(A2)}]{\textbf{(Modulating coefficient)}} The coefficient $\mathfrak{a}$ satisfies $\mathfrak{a} \in C^{1}_x(Q_1)\cap C^{1}_t(Q_1)$ and
\[
0 < \mathfrak{a}_{-} \coloneqq \inf_{Q_1} \mathfrak{a}(x,t) \leq \mathfrak{a}(x,t) \leq \mathfrak{a}_{+} \coloneqq \sup_{Q_1} \mathfrak{a}(x,t) < \infty,
\]
and
\[
\mathfrak{A}_0 \coloneqq \|D_{x,t}\mathfrak{a}\|_{L^\infty(Q_1)} < \infty.
\]
\item[{\bf(A3)}]{\textbf{(Source term)}} The source term satisfies $f \in L^\infty(Q_1)\cap C^0(Q_1)$.
\end{itemize}

We now state the main result.

\begin{theorem}\label{Thm01}
Assume \textbf{($\mathrm{A}1$)}--\textbf{($\mathrm{A}3$)}. Let $u \in C^0(Q_1)$ be a bounded viscosity solution to \eqref{Problem}. Then $u$ has a locally Hölder continuous gradient. Moreover, there exist constants
\[
\alpha = \alpha(p,\mathfrak{p},\mathfrak{q},n,\mathfrak{a}_{-},\mathfrak{a}_{+},\mathfrak{A}_0) \in \left(0,\frac{1}{1-\mathfrak{p}}\right),
\quad
\mathrm{C} = \mathrm{C}(p,\mathfrak{p},\mathfrak{q},n) > 0,
\]
such that
\[
\sup_{\substack{(x,t),(y,s)\in Q_{1/2} \\ (x,t)\neq (y,s)}}
\frac{|Du(x,t) - Du(y,s)|}{|x-y|^\alpha + |t-s|^{\alpha/2}}
\leq \mathrm{C}
\left(
\|u\|_{L^\infty(Q_1)}
+ [\mathfrak{a}]_{C^{0,1}_x(Q_1)}^{\frac{1}{1+\mathfrak{p}}}
+ \|f\|_{L^\infty(Q_1)}^{\frac{1}{1+\mathfrak{p}}}
\right),
\]
and
\[
\sup_{\substack{(x,t),(x,s)\in Q_{1/2} \\ t \neq s}}
\frac{|u(x,t) - u(x,s)|}{|t-s|^{\frac{1+\alpha}{2}}}
\leq \mathrm{C}
\left(
\|u\|_{L^\infty(Q_1)}
+ [\mathfrak{a}]_{C^{0,1}_x(Q_1)}^{\frac{1}{1+\mathfrak{p}}}
+ \|f\|_{L^\infty(Q_1)}^{\frac{1}{1+\mathfrak{p}}}
\right).
\]
\end{theorem}

Our result complements the estimates for doubly degenerate models obtained in \cite{BessaDaSilvaSa25} and extends, to the mixed-growth setting, the regularity theory developed by Fang--Zhang in \cite{FZ23}. In particular, it provides a further contribution to gradient regularity for nonhomogeneous problems of this type.

\subsubsection*{
Challenges and strategies for the proof of the Theorem \ref{Thm01}}

The proof of Theorem \ref{Thm01} combines the method of alternatives (cf. \cite{Attouchi20,AttouRuos20}) with an improvement-of-flatness scheme. The \textit{method of alternatives} is an iterative procedure yielding increasingly accurate affine approximations of solutions to \eqref{Problem} on intrinsically scaled cylinders. The iteration is governed by two competing regimes, the \textit{Degenerate Alternative} and the \textit{Smooth Alternative}, and terminates once the latter occurs.

More precisely, we show that there exist universal constants $\rho,\eta_{1}>0$ and $\theta\in(0,1)$, with $\rho<(1-\theta)^{1+\mathfrak{p}}$, such that, under the smallness condition
\[
\max\{[\mathfrak{a}]_{C^{0,1}_{x}(Q_{1})},\|f\|_{L^{\infty}(Q_{1})}\} \leq \eta_{1},
\]
one of the following alternatives holds:
\begin{enumerate}
\item[\checkmark] \textbf{Degenerate Alternative:} For every $j \in \mathbb{N}$, there exists $\ell_{j} \in \mathbb{R}^n$ with $|\ell_{j}| \leq \mathrm{C}_0(1-\theta)^{j}$ such that
\[
\operatornamewithlimits{osc}_{(x,t) \in Q^{\lambda_j}_{\rho^{j}}} \big( u(x,t) - \ell_{j} \cdot x \big)
\leq \eta_{1}\rho^{j}\lambda_j,
\]
where $\lambda_j := (1-\theta)^j$ and
\[
Q^{\lambda_j}_{\rho^{j}} := B_{\rho^{j}} \times \big(-\rho^{2j}\lambda_j^{-\mathfrak{p}},0\big].
\]
In particular, an improvement of flatness holds at all scales.

\item[\checkmark] \textbf{Smooth Alternative:} The iteration stops at some step $k_0$, namely $|\ell_{k_0}| \geq \mathrm{C}_0(1-\theta)^{k_0}$. In this case, the gradient of $u$ remains uniformly bounded away from zero in a suitable cylinder, allowing the use of classical regularity results for uniformly parabolic equations with smooth coefficients (see \cite[Theorem 1.1]{LSU68} and \cite[Lemma 12.13]{Lieberman96}).
\end{enumerate}

In contrast to the doubly degenerate case treated in \cite{BessaDaSilvaSa25}, no uniform spatial Lipschitz estimate is available for deviations of the form $u - \vec{q}\cdot x$. In that setting, writing
\[
u = \vec{q}\cdot x + (u - \vec{q}\cdot x),
\]
one may stop the iteration once the slope of the affine approximation exceeds the corresponding Lipschitz bound, thereby obtaining a positive lower bound for $|Du|$.

In the present framework, however, where singular behavior arises in the leading term of $\mathcal{H}$, such Lipschitz control is not preserved under these translations.

Accordingly, our approach requires a refined analysis of translated profiles—namely, functions obtained by subtracting affine functions—in the presence of mixed singular and degenerate terms. The proof proceeds as follows:
\begin{itemize}
\item[1.] We first apply the Jensen--Ishii method to derive spatial Lipschitz estimates for the translated equation \eqref{Problem} under smallness assumptions on $f$ and $[\mathfrak{a}]_{C^{0,1}_x}$ (Lemma \ref{LocalLiplemaprobtransl}).
\item[2.] In this regime, a compactness and approximation argument yields affine approximations in intrinsic cylinders (Lemma \ref{improvflat1}), leading to an iterative scheme on $\rho^{j}$-scaled cylinders (Lemma \ref{improvflat2}).
\item[3.] Finally, after a standard normalization and scaling reduction, we combine the method of alternatives with an improvement-of-flatness argument to complete the proof of Theorem \ref{Thm01}.
\end{itemize}

\noindent\textbf{Main strengths of the manuscript}

We emphasize several key features of our work that underscore its mathematical significance and distinguish it from recent contributions on models with nonhomogeneous degeneracy:

\begin{itemize}
\item[\checkmark] Analysis of a previously open regime involving simultaneous singular and degenerate behavior ($\mathfrak{p}<0$) in non-divergence form equations;
\item[\checkmark] Establishment of gradient H\"older regularity for viscosity solutions under minimal structural assumptions;
\item[\checkmark] Introduction of a novel analytical framework combining Jensen--Ishii techniques, intrinsic scaling, and a nonstandard improvement-of-flatness scheme (cf. \cite{AttouRuos20});
\item[\checkmark] Extension and unification of recent results, including homogeneous and doubly degenerate cases (cf. \cite{BessaDaSilvaSa25});
\item[\checkmark] Development of conceptual and technical tools applicable to broader classes of nonlinear parabolic problems (cf. \cite{DeF20} and \cite{FZ23}).
\end{itemize}

\vspace{0.3cm}

\noindent\textbf{Perspectives and future directions}

We highlight several directions that merit further investigation within this line of research:
\begin{itemize}
\item[\checkmark] Study of boundary regularity and global estimates in domains with low regularity;
\item[\checkmark] Investigation of fully nonlinear counterparts and stochastic interpretations;
\item[\checkmark] Derivation of sharp regularity results and optimal H\"older exponents within the mixed-growth regime;
\item[\checkmark] Applications to nonstandard diffusion models and to problems arising in image processing and game theory.
\end{itemize}

The paper is organized as follows. Section \ref{Sec2} collects preliminaries and auxiliary results. Section \ref{Sec3} establishes spatial Lipschitz regularity. Section \ref{Sec4} develops the method of alternatives and concludes the proof of Theorem \ref{Thm01}.


\section{Preliminaries}\label{Sec2}

In this section, we fix notation and recall basic definitions and auxiliary results used throughout the paper. For $x_0 \in \mathbb{R}^n$, $t_0 \in \mathbb{R}$, and $r>0$, we denote by $B_r(x_0)$ the open ball of radius $r$ centered at $x_0$, namely
\[
B_r(x_0) := \{x \in \mathbb{R}^n : |x - x_0| < r\}.
\]
We define the parabolic cylinder
\[
Q_r(x_0,t_0) := B_r(x_0) \times (t_0 - r^2, t_0],
\]
and the intrinsic (rescaled) cylinders
\[
Q_r^\lambda(x_0,t_0) := B_r(x_0) \times (t_0 - r^2 \lambda^{-\mathfrak{p}}, t_0],
\]
which account for the degeneracy in \eqref{Problem}. When $x_0 = 0$ and $t_0 = 0$, we omit the centers.

Given $T>0$ and an open set $\Omega \subset \mathbb{R}^n$, set $\Omega_T := \Omega \times (-T,0]$ and define the parabolic boundary
\[
\partial_{\mathrm{par}} \Omega_T := (\partial \Omega \times [-T,0)) \cup (\Omega \times \{t=-T\}).
\]
The parabolic distance between $P_1=(x,t)$ and $P_2=(y,s)$ is defined by
\[
d(P_1,P_2) := |x-y| + |t-s|^{1/2}.
\]

For parabolic Hölder spaces, let $\alpha \in (0,1]$. We define
\[
[u]_{C^{\alpha,\alpha/2}(Q_r)} := \sup_{\substack{(x,t),(y,s)\in Q_r \\ (x,t)\neq (y,s)}} 
\frac{|u(x,t)-u(y,s)|}{d((x,t),(y,s))^\alpha},
\]
\[
[u]_{C^{0,\alpha}_x(Q_r)} := \sup_{\substack{(x,t),(y,t)\in Q_r \\ x\neq y}} 
\frac{|u(x,t)-u(y,t)|}{|x-y|^\alpha}, 
\quad
[u]_{C^{0,\alpha}_t(Q_r)} := \sup_{\substack{(x,t),(x,s)\in Q_r \\ t\neq s}} 
\frac{|u(x,t)-u(x,s)|}{|t-s|^\alpha},
\]
and
\[
\|u\|_{C^{\alpha,\alpha/2}(Q_r)} := \|u\|_{L^\infty(Q_r)} + [u]_{C^{\alpha,\alpha/2}(Q_r)}.
\]

The space $C^{1+\alpha,(1+\alpha)/2}(Q_r)$ consists of functions with finite norm
\[
\|u\|_{C^{1+\alpha,(1+\alpha)/2}(Q_r)} 
:= \|u\|_{L^\infty(Q_r)} + \|Du\|_{L^\infty(Q_r)} + [u]_{C^{1+\alpha,(1+\alpha)/2}(Q_r)},
\]
where
\[
[u]_{C^{1+\alpha,(1+\alpha)/2}(Q_r)} 
:= [Du]_{C^{\alpha,\alpha/2}(Q_r)} 
+ \sup_{\substack{(x,t),(x,s)\in Q_r \\ t\neq s}} 
\frac{|u(x,t)-u(x,s)|}{|t-s|^{(1+\alpha)/2}}.
\]

We next recall the notion of viscosity solution used throughout (cf. \cite{Attouchi20,AttouRuos20}).

\begin{definition}[Viscosity solution]
A locally bounded upper (respectively, lower) semicontinuous function $u$ in $Q_1$ is called a viscosity subsolution (respectively, supersolution) of \eqref{Problem} if, for every $(x_0,t_0)\in Q_1$, one of the following holds:
\begin{enumerate}
\item[(i)] For every $\phi \in C^{2,1}(Q_1)$ such that $u-\phi$ attains a local maximum (respectively, minimum) at $(x_0,t_0)$ and $D\phi(x_0,t_0)\neq 0$,
\[
\partial_t \phi(x_0,t_0) - \mathcal{H}(x_0,t_0,D\phi(x_0,t_0)) \Delta_p^{\mathrm{N}}\phi(x_0,t_0) \leq f(x_0,t_0),
\]
\[
\text{(respectively, } \partial_t \phi(x_0,t_0) - \mathcal{H}(x_0,t_0,D\phi(x_0,t_0)) \Delta_p^{\mathrm{N}}\phi(x_0,t_0) \geq f(x_0,t_0)\text{)}.
\]
\item[(ii)] There exist $\delta_1>0$ and $\phi \in C^{1}((t_0-\delta_1,t_0+\delta_1))$ such that
\[
\begin{cases}
\phi(t_0)=0,\\
u(x_0,t_0) \geq u(x_0,t) - \phi(t), \quad \forall t \in (t_0-\delta_1,t_0+\delta_1),\\
\displaystyle \sup_{t \in (t_0-\delta_1,t_0+\delta_1)} (u(x,t)-\phi(t)) \text{ is constant near } x_0,
\end{cases}
\]
respectively,
\[
\begin{cases}
\phi(t_0)=0,\\
u(x_0,t_0) \leq u(x_0,t) - \phi(t), \quad \forall t \in (t_0-\delta_1,t_0+\delta_1),\\
\displaystyle \inf_{t \in (t_0-\delta_1,t_0+\delta_1)} (u(x,t)-\phi(t)) \text{ is constant near } x_0,
\end{cases}
\]
and
\[
\phi'(t_0) \leq f(x_0,t_0) \quad \text{(respectively, } \phi'(t_0) \geq f(x_0,t_0)\text{)}.
\]
\end{enumerate}
A continuous function $u$ is a viscosity solution if it is both a subsolution and a supersolution.
\end{definition}

We also introduce the parabolic sub- and superjets of a function $v$ at $(x,t)$:
\[
\mathcal{J}^{\pm}v(x,t)
:= \left\{
\begin{aligned}
(\tau,\xi,X)&\in\mathbb{R}\times\mathbb{R}^n\times\mathrm{Sym}(n):(\tau,\xi,X)=(\phi_t,D\phi,D^2\phi)(x,t),\\
\phi&\in C^{2,1},\ \text{touches } v \text{ from below (above) at }(x,t)
\end{aligned}
\right\}.
\]
and the corresponding closures
\[
\overline{\mathcal{J}}^{\pm}v(x,t)
:= \left\{
\begin{aligned} (\tau,\xi,X): \exists (x_k,t_k)\to(x,t),\ 
(\tau_k,\xi_k,X_k)\in \mathcal{J}^{\pm}v(x_k,t_k),\\
(\tau_k,\xi_k,X_k)\to(\tau,\xi,X),\
v(x_k,t_k)\to v(x,t)
\end{aligned}\right\}.
\]

We conclude by recalling Jensen--Ishii's lemma, a key tool for compactness arguments in nonlinear PDEs; see \cite[Theorem 12.2]{Crandall} and \cite[Theorem 8.3]{CrandallIshiiLions}.

\begin{lemma}[\bf Jensen--Ishii's Lemma]\label{JensenIshii}
Let $u_i$ be an upper semicontinuous function in $Q_1$ for $i=1,\ldots,k$. Let $\varphi$ be defined on $(B_1)^k \times (-1,0)$ such that the function 
\[
(x_1,\dots,x_k,t) \mapsto \varphi(x_1,\dots,x_k,t)
\]
is once continuously differentiable in $t$ and twice continuously differentiable in $(x_1,\dots,x_k) \in (B_1)^k$. Suppose that 
\[
w(x_1,\dots,x_k,t) \coloneqq u_1(x_1,t) + \cdots + u_k(x_k,t) - \varphi(x_1,\dots,x_k,t)
\]
attains a local maximum at $(\bar{x}_1,\dots,\bar{x}_k,\bar{t}) \in (B_1)^k \times (-1,0)$. Assume, moreover, that there exists $r > 0$ such that for every $\mathrm{M}_{\star} > 0$ there exists a constant $\mathrm{C}_{\star}$ such that, for $i=1,\ldots,k$,
\[
b_i \leq \mathrm{C}_{\star} \quad \text{whenever } (b_i,\vec{q}_i,X_i) \in \mathcal{J}^{+}u_i(x_i,t),
\]
\[
|x_i - \bar{x}_i| + |t-\bar{t}| \leq r, \quad \text{and} \quad |u_i(x_i,t)| + |\vec{q}_i| + \|\mathrm{X}_i\| \leq \mathrm{C}_{\star}.
\]
Then, for each $\varepsilon > 0$, there exist $\mathrm{X}_i \in \mathrm{Sym}(n)$ such that:
\begin{itemize}
\item[(i)] $(b_i, D_{x_i}\varphi(\bar{x}_1,\ldots,\bar{x}_k),\mathrm{X}_i) \in \overline{\mathcal{J}}^{+}u_i(\bar{x}_i,\bar{t})$ for $i=1,\dots,k$,
\item[(ii)] $- \left( \frac{1}{\varepsilon} + \|\mathrm{A}\| \right) \mathrm{Id}_n \leq \begin{pmatrix} \mathrm{X}_1 & & 0 \\ & \ddots & \\ 0 & & \mathrm{X}_k \end{pmatrix} \leq \mathrm{A} + \varepsilon \mathrm{A}^2$,
\item[(iii)] $b_1 + \cdots + b_k = \varphi_t(\bar{x}_1,\ldots,\bar{x}_k,\bar{t}),$
\end{itemize}
where $\mathrm{A} = D^2 \varphi(\bar{x}_1,\ldots,\bar{x}_k,\bar{t})$.
\end{lemma}

We conclude this section by recalling a low-regularity estimate for \eqref{Problem}, namely Lipschitz continuity in the spatial variables and Hölder continuity in time (see \cite[Lemmas 4.3 and 4.5]{BessaDaSilvaSa25}).

\begin{lemma}[{\bf Lower regularity}]\label{LocalLiplemaprob}
Let $u$ be a bounded viscosity solution to \eqref{Problem} in $Q_{1}$. Suppose that structural conditions \(\mathbf{(A1)-(A3)}\) are satisfies. Then,  
\begin{itemize}
\item[(a)] $u\in C^{0,1}(Q_{\frac{7}{8}})$ and there exists a constant $\mathrm{C}=\mathrm{C}\left(n,p,\mathfrak{p},\mathfrak{q},\mathfrak{a}_{-}\right)>0$ such that
\begin{eqnarray*}
[u]_{C^{0,1}_{x}(Q_{\frac{7}{8}})}\leq  \mathrm{C}\left(\|u\|_{L^{\infty}(Q_{1})}+\|u\|_{L^{\infty}(Q_{1})}^{\frac{1}{1+\mathfrak{p}}}+[\mathfrak{a}]_{C^{0,1}_{x}(Q_{1})}+\|f\|_{L^{\infty}(Q_{1})}^{\frac{1}{1+\mathfrak{p}}}\right),
\end{eqnarray*}

\item [(b)] $u\in C^{0,\frac{1+\mathfrak{p}}{2+\mathfrak{p}+\mathfrak{q}}}(Q_{\frac{3}{4}})$ and
\begin{eqnarray*}
[u]_{C^{0,\frac{1+\mathfrak{p}}{2+\mathfrak{p}+\mathfrak{q}}}_{t}(Q_{\frac{3}{4}})}\leq  \mathrm{C},
\end{eqnarray*}
where $\mathrm{C}=\mathrm{C}\left(n,p,\mathfrak{p},\mathfrak{q},\mathfrak{a}_{-},\|u\|_{L^{\infty}(Q_{1})}[\mathfrak{a}]_{C^{0,1}_{x}(Q_{1})},\|f\|_{L^{\infty}(Q_{1})}\right)>0$.
\end{itemize}

\end{lemma}

\section{The translated problem: spatial Lipschitz regularity}\label{Sec3}

As part of our approach, we establish a precompactness property for solutions to the translated equation
\begin{equation}\label{translatedproblem}
 w_{t} - \mathcal{H}(x,t,Dw+\vec{q}) \Delta_{p}^{\mathrm{N}}(w+\vec{q}\cdot x)=\tilde{f}(x,t)\quad \text{in}\quad Q_1,
\end{equation}
where $\vec{q}\in\mathbb{R}^{n}$.

Let $w$ be a viscosity solution to \eqref{translatedproblem}. Under suitable bounds on $\|w\|_{L^{\infty}(Q_{1})}$ and $\|\tilde{f}\|_{L^{\infty}(Q_{1})}$, and depending on $|\vec{q}|$, we derive a spatial Lipschitz estimate. More precisely, if
\[
\|w\|_{L^{\infty}(Q_{1})}\leq 1,\quad [\mathfrak{a}]_{C^{0,1}_{x}(Q_{1})}\leq 1,\quad \|\tilde{f}\|_{L^{\infty}(Q_{1})}\leq 1,\quad |\vec{q}|\geq 1,
\]
then, applying Lemma~\ref{LocalLiplemaprob} to
\[
\tilde{w}(x,t)=w(x,t)+\vec{q}\cdot x,
\]
which solves \eqref{Problem}, we obtain
\begin{eqnarray}
|w(x,t)-w(y,t)|
&\leq& |\tilde{w}(x,t)-\tilde{w}(y,t)| + |\vec{q}|\,|x-y| \nonumber\\
&\leq& \mathrm{C}\big(1+|\vec{q}|+|\vec{q}|^{\frac{1}{1+\mathfrak{p}}}\big)|x-y| \nonumber\\
&\leq& \bar{\mathrm{C}}_{0}|\vec{q}|^{\frac{1}{1+\mathfrak{p}}}|x-y|, \label{est3.2}
\end{eqnarray}
for all $x,y\in Q_{7/8}$ and $t\in (-(7/8)^2,0]$, where $\bar{\mathrm{C}}_{0}>0$ depends only on $n$, $p$, $\mathfrak{p}$, $\mathfrak{q}$, and $\mathfrak{a}_{-}$.

The main result of this section sharpens \eqref{est3.2} under an additional restriction on $|\vec{q}|$.

\begin{lemma}[\bf Lipschitz regularity in space]\label{LocalLiplemaprobtransl}
Assume that $\mathbf{(A1)}$--$\mathbf{(A3)}$ hold. Let $\vec{q}\in\mathbb{R}^{n}$ satisfy $2\leq |\vec{q}|\leq \mathrm{K}_{0}$, for some $\mathrm{K}_{0}=\mathrm{K}_{0}(n,p,\mathfrak{p},\mathfrak{q},\mathfrak{a}_{-})>0$, and let $w$ be a viscosity solution to \eqref{translatedproblem}. There exists $\eta_{1}>0$, depending only on $n$, $p$, $\mathfrak{p}$, $\mathfrak{q}$, and $\mathfrak{a}_{-}$, such that if
\[
\max\{\|w\|_{L^{\infty}(Q_{1})},[\mathfrak{a}]_{C^{0,1}_{x}(Q_{1})},\|\tilde{f}\|_{L^{\infty}(Q_{1})}\}\leq \eta_{1},
\]
then $w$ is locally Lipschitz in space and
\[
|w(x,t)-w(y,t)|\leq |x-y|
\]
for all $(x,t),(y,t)\in Q_{3/4}$.
\end{lemma}

To prove this result, we first establish spatial Hölder continuity for solutions to \eqref{translatedproblem}, with exponent $\nu\in(0,1)$ depending on $|\vec{q}|$ and \eqref{est3.2}.

\begin{lemma}[\bf Hölder regularity in space]\label{Localholderlemaprobtransl}
Assume that $\mathbf{(A1)}$--$\mathbf{(A3)}$ hold. Let $\vec{q}\in\mathbb{R}^{n}$ satisfy $2\leq |\vec{q}|\leq \mathrm{K}_{0}$, for some $\mathrm{K}_{0}=\mathrm{K}_{0}(n,p,\mathfrak{p},\mathfrak{q},\mathfrak{a}_{-})>0$, and let $w$ be a viscosity solution to \eqref{translatedproblem}. There exists $\eta_{0}>0$, depending only on $n$, $p$, $\mathfrak{p}$, $\mathfrak{q}$, and $\mathfrak{a}_{-}$, such that if
\[
\max\{\|w\|_{L^{\infty}(Q_{1})},[\mathfrak{a}]_{C^{0,1}_{x}(Q_{1})},\|\tilde{f}\|_{L^{\infty}(Q_{1})}\}\leq \eta_{0},
\]
then $w$ is locally $\nu$-Hölder continuous in space and
\[
|w(x,t)-w(y,t)|\leq |x-y|^{\nu}
\]
for all $(x,t),(y,t)\in Q_{7/8}$, where
\[
\nu=\min\left\{\frac{1}{2},\frac{4}{15\bar{\mathrm{C}}_{0}}\mathrm{K}_{0}^{\frac{\mathfrak{p}}{1+\mathfrak{p}}}\right\}.
\]
Here, $\bar{\mathrm{C}}_{0}$ is the constant in \eqref{est3.2}.
\end{lemma}


\begin{proof}
We fix $x_{0},y_{0}\in B_{\frac{7}{8}}$ and $t_{0}\in\left(-\left(\frac{7}{8}\right)^{2},0\right)$. For suitable constants $\mathfrak{L}_{1},\mathfrak{L}_{2}>0$, define the auxiliary function
\[
\phi(x,y,t)=w(x,t)-w(y,t)-\mathfrak{L}_{2}|x-y|^{\nu}-\frac{\mathfrak{L}_{1}}{2}\left(|x-x_{0}|^{2}+|y-y_{0}|^{2}+(t-t_{0})^{2}\right).
\]
We claim that $\phi \leq 0$ in $\overline{B_{\frac{7}{8}}}\times \overline{B_{\frac{7}{8}}}\times \left[-\left(\frac{7}{8}\right)^{2},0\right]$. We argue by contradiction.

Assume that $\phi$ attains a positive maximum at some point $(\bar{x},\bar{y},\bar{t})$ in $\overline{B_{\frac{7}{8}}}\times \overline{B_{\frac{7}{8}}}\times \left[-\left(\frac{7}{8}\right)^{2},0\right]$. Since $\phi(\bar{x},\bar{y},\bar{t})>0$, it follows that $\bar{x}\neq \bar{y}$. 

Next, choose
\begin{eqnarray*}
\mathfrak{L}_{1}=\frac{32\|w\|_{L^{\infty}(Q_{1})}}{\left(\min\{\operatorname{dist}((x_{0},t_{0}),\partial Q_{7/8}),\operatorname{dist}((y_{0},t_{0}),\partial Q_{7/8})\}\right)^{2}}=\mathrm{C}\|w\|_{L^{\infty}(Q_{1})}.
\end{eqnarray*}
Then
\begin{eqnarray*}
|\bar{x}-x_{0}|+|\bar{t}-t_{0}|\leq \frac{\operatorname{dist}((x_{0},t_{0}),\partial Q_{7/8})}{2}
\end{eqnarray*}
and
\begin{eqnarray*}
|\bar{y}-y_{0}|+|\bar{t}-t_{0}|\leq \frac{\operatorname{dist}((y_{0},t_{0}),\partial Q_{7/8})}{2}.
\end{eqnarray*}
Consequently, $\bar{x},\bar{y}\in B_{\frac{7}{8}}$ and $\bar{t}\in \left(-\left(\frac{7}{8}\right)^{2},0\right)$.

Moreover, by the Lipschitz estimate in \eqref{est3.2}, we obtain
\begin{equation*}
\mathfrak{L}_{2}|\bar{x}-\bar{y}|^{\nu}\leq\bar{\mathrm{C}}_{0}|\vec{q}|^{\frac{1}{1+\mathfrak{p}}}|\bar{x}-\bar{y}|.
\end{equation*}
Hence, by the definition of $\nu$,
\begin{equation}\label{est3.3}
3\mathfrak{L}_{2}\nu|\bar{x}-\bar{y}|^{\nu-1}\leq \frac{4}{5}|\vec{q}|,
\end{equation}
since $\nu\bar{\mathrm{C}}_{0}\leq \frac{4}{15}\mathrm{K}_{0}^{\frac{\mathfrak{p}}{1+\mathfrak{p}}}$ and $|\vec{q}|\leq \mathrm{K}_{0}$.

With these preparations, we invoke Jensen--Ishii's Lemma \ref{JensenIshii} for the functions
\[
\tilde{u}(x,t)= w(x,t)-\frac{\mathfrak{L}_{1}}{2}|x-x_{0}|^{2}-\frac{\mathfrak{L}_{1}}{2}(t-t_{0})^{2}, 
\quad  
\tilde{v}(y,t)=-w(y,t)-\frac{\mathfrak{L}_{1}}{2}|y-y_{0}|^{2},
\]
and deduce the existence of
\begin{eqnarray*}
\left(\tau+\mathfrak{L}_{1}(\bar{t}-t_{0}),\zeta_{1},\mathrm{X}+\mathfrak{L}_{1}\mathrm{Id}_n\right)\in\overline{\mathcal{J}}^{+}(w)(\bar{x},\bar{t})
\end{eqnarray*}
and
\begin{eqnarray*}
\left(\tau,\zeta_{2},\mathrm{Y}-\mathfrak{L}_{1}\mathrm{Id}_n\right)\in\overline{\mathcal{J}}^{-}(w)(\bar{y},\bar{t}),
\end{eqnarray*}
where
\begin{eqnarray*}
\zeta_{1}=\mathfrak{L}_{2}\nu|\bar{x}-\bar{y}|^{\nu-1}\frac{\bar{x}-\bar{y}}{|\bar{x}-\bar{y}|}+\mathfrak{L}_{1}(\bar{x}-x_{0}),
\end{eqnarray*}
and
\begin{eqnarray*}
\zeta_{2}=\mathfrak{L}_{2}\nu|\bar{x}-\bar{y}|^{\nu-1}\frac{\bar{x}-\bar{y}}{|\bar{x}-\bar{y}|}-\mathfrak{L}_{1}(\bar{y}-y_{0}).
\end{eqnarray*}

Imposing $\mathfrak{L}_{2}>\frac{2\mathfrak{L}_{1}}{\nu}$, we obtain
\begin{eqnarray}
\frac{\mathfrak{L}_{2}}{2}\nu|\bar{x}-\bar{y}|^{\nu-1}\leq |\zeta_{i}|\leq 2\mathfrak{L}_{2}\nu|\bar{x}-\bar{y}|^{\nu-1}, \quad i=1,2. \label{est3.4}
\end{eqnarray}

Furthermore, Jensen--Ishii's Lemma yields that for all $\kappa>0$ such that $\kappa\mathrm{Z}<\mathrm{Id}_{n}$,
\begin{equation}\label{est3.5}
-\frac{2}{\kappa}\begin{pmatrix} \mathrm{Id}_n & 0 \\ 0 & \mathrm{Id}_n \end{pmatrix}
\leq 
\begin{pmatrix} \mathrm{X} & 0 \\ 0 & -\mathrm{Y} \end{pmatrix}
\leq 
\begin{pmatrix} \mathrm{Z}^{\kappa} & -\mathrm{Z} \\ -\mathrm{Z} & \mathrm{Z}^{\kappa} \end{pmatrix},
\end{equation}
where
\begin{align*}
\mathrm{Z} &=\mathfrak{L}_{2} \nu |\bar{x} - \bar{y}|^{\nu - 2} \left( \mathrm{Id}_n + (\nu - 2) \frac{\bar{x} - \bar{y}}{|\bar{x} - \bar{y}|} \otimes \frac{\bar{x} - \bar{y}}{|\bar{x} - \bar{y}|} \right),
\end{align*}
and
\[
\mathrm{Z}^{\kappa}=(\mathrm{Id}_n-\kappa \mathrm{Z})^{-1}\mathrm{Z}.
\]

Choosing $\kappa = \dfrac{1}{2\mathfrak{L}_{2}\nu|\bar{x}-\bar{y}|^{\nu-2}}$ and invoking the Sherman--Morrison formula (see \cite{SherMorr50}), we conclude that
\begin{eqnarray*}
\mathrm{Z}^{\kappa}=2\mathfrak{L}_{2}\nu|\bar{x}-\bar{y}|^{\nu-2}\left(\mathrm{Id}_n-2\frac{2-\nu}{3-\nu}\frac{\bar{x} - \bar{y}}{|\bar{x} - \bar{y}|} \otimes \frac{\bar{x} - \bar{y}}{|\bar{x} - \bar{y}|} \right).
\end{eqnarray*}

In particular, for $\xi=\frac{\bar{x}-\bar{y}}{|\bar{x}-\bar{y}|}$,
\begin{equation}\label{est3.6}
\langle \mathrm{Z}^{\kappa}\xi,\xi\rangle
=2\mathfrak{L}_{2}\nu|\bar{x}-\bar{y}|^{\nu-2}\left(\frac{\nu-1}{3-\nu}\right)<0.
\end{equation}
Applying the matrix inequality \eqref{est3.5} to the pair $(\xi,\xi)$ with $|\xi|=1$, we deduce that $\mathrm{X}-\mathrm{Y}\leq 0$ and
\begin{equation}\label{est3.7}
\|\mathrm{X}\|,\|\mathrm{Y}\|\leq 4\mathfrak{L}_{2}\nu|\bar{x}-\bar{y}|^{\nu-2}.
\end{equation}

For $i=1,2$, set $\xi_{i} := \zeta_{i} + \vec{q}$. Combining \eqref{est3.3} and \eqref{est3.4}, we obtain
\begin{eqnarray}\label{est3.8}
\frac{3}{4}\mathfrak{L}_{2}\nu|\bar{x}-\bar{y}|^{\nu-1}
\leq |\vec{q}|-|\zeta_{i}|
\leq |\xi_{i}|
\leq |\vec{q}|+|\zeta_{i}|
\leq 2|\vec{q}|,
\quad \text{and} \quad
|\xi_{i}|\geq \frac{1}{5}|\vec{q}|.
\end{eqnarray}

For $\eta\in\mathbb{R}^{n}$, define
\begin{eqnarray*}
\mathcal{A}(\eta)=\mathrm{Id}_n+(p-2)\frac{\eta}{|\eta|}\otimes \frac{\eta}{|\eta|}.
\end{eqnarray*}
The eigenvalues of $\mathcal{A}(\eta)$ lie in the interval $[\min\{1,p-1\},\max\{1,p-1\}]$ (see \cite[Lemma 5.1]{AttouRuos20}).

Using the sub- and superdifferential inequalities, we obtain
\[
\tau + \mathfrak{L}_{1}(\bar{t} - t_0) - \mathcal{H}(\bar{x},\bar{t},\xi_{1}) \operatorname{tr}\!\left(\mathcal{A}(\xi_1)(\mathrm{X} + \mathfrak{L}_{1} \mathrm{Id}_n)\right) \leq \tilde{f}(\bar{x},\bar{t}),
\]
and
\[
-\tau + \mathcal{H}(\bar{y},\bar{t},\xi_{2}) \operatorname{tr}\!\left(\mathcal{A}(\xi_2)(\mathrm{Y} - \mathfrak{L}_{1} \mathrm{Id}_n)\right) \leq \tilde{f}(\bar{y},\bar{t}).
\]
Combining these inequalities yields
\begin{eqnarray*}
\mathfrak{L}_{1}(\bar{t} - t_0)
&\leq& \mathcal{H}(\bar{x},\bar{t},\xi_{1}) \operatorname{tr}\!\left(\mathcal{A}(\xi_1)(\mathrm{X} + \mathfrak{L}_{1} \mathrm{Id}_n)\right)\\
&&-\mathcal{H}(\bar{y},\bar{t},\xi_{2}) \operatorname{tr}\!\left(\mathcal{A}(\xi_2)(\mathrm{Y} - \mathfrak{L}_{1} \mathrm{Id}_n)\right)
+2 \|\tilde{f}\|_{L^{\infty}(Q_{1})}.
\end{eqnarray*}

By the definition of $\mathcal{H}$, the bound $|\bar{t}-t_{0}|\leq 2$, and rearranging terms, we infer
\begin{eqnarray}
0&\leq& \frac{2(\mathfrak{L}_{1}+\|\tilde{f}\|_{L^{\infty}(Q_{1})})}{|\xi_{1}|^{\mathfrak{p}}}
+\mathfrak{L}_{1}\left[\operatorname{tr}(\mathcal{A}(\xi_{1}))+\frac{|\xi_{2}|^{\mathfrak{p}}\operatorname{tr}(\mathcal{A}(\xi_{2}))}{|\xi_{1}|^{\mathfrak{p}}}\right]\nonumber\\
&&+\operatorname{tr}(\mathcal{A}(\xi_{1})(\mathrm{X}-\mathrm{Y}))
+\operatorname{tr}((\mathcal{A}(\xi_{1})-\mathcal{A}(\xi_{2}))\mathrm{Y})
+\frac{(|\xi_{1}|^{\mathfrak{p}}-|\xi_{2}|^{\mathfrak{p}})\operatorname{tr}(\mathcal{A}(\xi_{2})\mathrm{Y})}{|\xi_{1}|^{\mathfrak{p}}}
\nonumber\\
&&+(\mathfrak{a}(\bar{x},\bar{t})-\mathfrak{a}(\bar{y},\bar{t}))
\frac{|\xi_{2}|^{\mathfrak{q}}}{|\xi_{1}|^{\mathfrak{p}}}
\left(\operatorname{tr}(\mathcal{A}(\xi_{2})\mathrm{Y})-\mathfrak{L}_{1}\operatorname{tr}(\mathcal{A}(\xi_{2}))\right)\nonumber\\
&&+\mathfrak{a}(\bar{x},\bar{t})\Bigg[
\frac{|\xi_{1}|^{\mathfrak{q}}\left(\operatorname{tr}(\mathcal{A}(\xi_{1})\mathrm{X})-\mathfrak{L}_{1}\operatorname{tr}(\mathcal{A}(\xi_{1}))\right)}{|\xi_{1}|^{\mathfrak{p}}}
\nonumber\\
&&\qquad\qquad
-\frac{|\xi_{2}|^{\mathfrak{q}}\left(\operatorname{tr}(\mathcal{A}(\xi_{2})\mathrm{Y})-\mathfrak{L}_{1}\operatorname{tr}(\mathcal{A}(\xi_{2}))\right)}{|\xi_{1}|^{\mathfrak{p}}}
\Bigg]\nonumber\\
&\eqcolon& \frac{2(\mathfrak{L}_{1}+\|\tilde{f}\|_{L^{\infty}(Q_{1})})}{|\xi_{1}|^{\mathfrak{p}}}
+\mathcal{I}_{1}+\mathcal{I}_{2}+\mathcal{I}_{3}+\mathcal{I}_{4}+\mathcal{I}_{5}+\mathcal{I}_{6}.
\label{est3.9}
\end{eqnarray}

Without loss of generality, we may assume that $|\xi_{1}|\geq |\xi_{2}|$ (the complementary case is analogous). We now estimate each term $\mathcal{I}_{j}$, $j=1,\dots,6$.\\

\medskip

\textbf{Estimate of $\mathcal{I}_{1}$:} Since the eigenvalues of $\mathcal{A}(\xi_{i})$ lie in the interval $[\min\{1, p-1\},\max\{1,p-1\}]$, we obtain
\begin{equation}\label{est3.10}
\mathcal{I}_{1}\leq \mathfrak{L}_{1}n\max\{1,p-1\}\left(1+\frac{|\xi_{2}|^{\mathfrak{p}}}{|\xi_{1}|^{\mathfrak{p}}}\right)\leq 10^{\mathfrak{p}}n\max\{1,p-1\}\mathfrak{L}_{1},
\end{equation}
where, by \eqref{est3.8},
\begin{equation}\label{est3.11}
\frac{1}{10}\leq\frac{|\xi_{2}|}{|\xi_{1}|}\leq 10.
\end{equation}

\medskip

\noindent\textbf{Estimate of $\mathcal{I}_{2}$:} Let $\lambda_{i}(\mathrm{M})$ denote the $i$-th eigenvalue of a matrix $\mathrm{M}$. Applying \eqref{est3.5} to $(\xi,-\xi)$ with $\xi=\frac{\bar{x}-\bar{y}}{|\bar{x}-\bar{y}|}$, we obtain
\[
\langle(\mathrm{X}-\mathrm{Y})\xi,\xi\rangle
\leq 4\langle \mathrm{Z}^{\kappa}\xi,\xi\rangle
\leq 8\mathfrak{L}_{2}\nu|\bar{x}-\bar{y}|^{\nu-2}\left(\frac{\nu-1}{3-\nu}\right)<0.
\]
Hence, there exists $i_{0}\in\{1,\ldots,n\}$ such that
\[
\lambda_{i_{0}}(\mathrm{X}-\mathrm{Y})
\leq 8\mathfrak{L}_{2}\nu|\bar{x}-\bar{y}|^{\nu-2}\left(\frac{\nu-1}{3-\nu}\right).
\]
Consequently,
\begin{eqnarray}
\mathcal{I}_{2}
&\leq& \sum_{i=1}^{n}\lambda_{i}(\mathcal{A}(\xi_{1}))\lambda_{i}(\mathrm{X}-\mathrm{Y})\nonumber\\
&\leq& \lambda_{i_{0}}(\mathcal{A}(\xi_{1}))\lambda_{i_{0}}(\mathrm{X}-\mathrm{Y})\nonumber\\
&\leq& 8\mathfrak{L}_{2}\min\{1,p-1\}\nu|\bar{x}-\bar{y}|^{\nu-2}\left(\frac{\nu-1}{3-\nu}\right),
\label{est3.12}
\end{eqnarray}
since $\mathrm{X}-\mathrm{Y}\leq 0$ and $\lambda_{i}(\mathcal{A}(\xi_{1}))\in [\min\{1,p-1\},\max\{1,p-1\}]$.

\medskip

\noindent\textbf{Estimate of $\mathcal{I}_{3}$:} Let $\hat{\eta}=\frac{\eta}{|\eta|}$ for $\eta\in\mathbb{R}^{n}\setminus\{0\}$. Then
\begin{equation*}
\mathcal{I}_{3}\leq 2n|p-2|\|\mathrm{Y}\||\hat{\xi}_{1}-\hat{\xi}_{2}|
\stackrel{\eqref{est3.7}}{\leq}
8n|p-2|\mathfrak{L}_{2}\nu|\bar{x}-\bar{y}|^{\nu-2}|\hat{\xi}_{1}-\hat{\xi}_{2}|.
\end{equation*}
Using $|\xi_{1}-\xi_{2}|\leq \frac{7}{2}\mathfrak{L}_{1}$ and \eqref{est3.8}, we infer
\[
|\hat{\xi}_{1}-\hat{\xi}_{2}|
\leq \max\left\{\frac{|\xi_{1}-\xi_{2}|}{|\xi_{1}|},\frac{|\xi_{1}-\xi_{2}|}{|\xi_{2}|}\right\}
\leq \frac{14\mathfrak{L}_{1}}{3\mathfrak{L}_{2}\nu|\bar{x}-\bar{y}|^{\nu-1}}.
\]
Consequently,
\begin{equation}\label{est3.13}
\mathcal{I}_{3}\leq 40n|p-2|\mathfrak{L}_{1}|\bar{x}-\bar{y}|^{-1}.
\end{equation}

\medskip

\noindent\textbf{Estimate of $\mathcal{I}_{4}$:} By the mean value theorem,
\begin{equation*}
\frac{\big||\xi_{1}|^{\mathfrak{p}}-|\xi_{2}|^{\mathfrak{p}}\big|}{|\xi_{1}|^{\mathfrak{p}}}
\leq C|\xi_{1}-\xi_{2}|\frac{1}{|\xi_{1}|}
\leq C\mathfrak{L}_{1}\left(\mathfrak{L}_{2}\nu|\bar{x}-\bar{y}|^{\nu-1}\right)^{-1},
\end{equation*}
where we used \eqref{est3.8}. Hence,
\begin{equation}
\mathcal{I}_{4}\leq C(1+|p-2|)\mathfrak{L}_{1}\mathfrak{L}_{2}^{\mathfrak{p}}\nu^{\mathfrak{p}}|\bar{x}-\bar{y}|^{(\nu-1)\mathfrak{p}-1}.
\label{est3.14}
\end{equation}

\medskip

\noindent\textbf{Estimate of $\mathcal{I}_{5}$:} Using the Lipschitz continuity of $\mathfrak{a}$ in space and \eqref{est3.11}, we obtain
\begin{eqnarray}
\mathcal{I}_{5}
&\leq& C|\bar{x}-\bar{y}|\,|\xi_{1}|^{\mathfrak{q}-\mathfrak{p}}
[\mathfrak{a}]_{C^{0,1}_{x}(Q_{1})}
\left(|\operatorname{tr}(\mathcal{A}(\xi_{2})\mathrm{Y})|+\mathfrak{L}_{1}|\operatorname{tr}(\mathcal{A}(\xi_{2}))|\right)\nonumber\\
&\leq& C[\mathfrak{a}]_{C^{0,1}_{x}(Q_{1})}
\big(
\mathfrak{L}_{2}^{1+\mathfrak{p}}\nu^{1+\mathfrak{p}}|\bar{x}-\bar{y}|^{(\nu-1)(1+\mathfrak{p})}
+\mathfrak{L}_{1}\mathfrak{L}_{2}^{\mathfrak{p}}\nu^{\mathfrak{p}}|\bar{x}-\bar{y}|^{(\nu-1)\mathfrak{p}+1}
\big),
\label{est3.15}
\end{eqnarray}
where we used \eqref{est3.7} and \eqref{est3.8}.

\medskip

\noindent\textbf{Estimate of $\mathcal{I}_{6}$:} Arguing as in \eqref{est3.10}--\eqref{est3.14}, it is straightforward to verify that
\begin{eqnarray}
\mathcal{I}_{6}&\leq& \mathfrak{a}(x,t)(2\mathrm{K}_{0})^{\mathfrak{q}-\mathfrak{p}}\Bigg(8\mathfrak{L}_{2}\min\{1,p-1\}\nu|\bar{x}-\bar{y}|^{\nu-2}\left(\frac{\nu-1}{3-\nu}\right)\nonumber\\
&+&10^{-\mathfrak{p}}\!\left(1+\max\{10^{-\mathfrak{q}},10^{\mathfrak{q}}\}(2\mathrm{K}_{0})^{-\mathfrak{p}}\right)n\max\{1,p-1\}\mathfrak{L}_{1}\nonumber\\
&+&40n|p-2|\mathfrak{L}_{1}|\bar{x}-\bar{y}|^{-1}\nonumber\\
&+&4^{2-\mathfrak{p}}(1+|p-2|)\mathfrak{L}_{1}^{-\mathfrak{p}}\mathfrak{L}_{2}^{1+\mathfrak{p}}
\nu^{1+\mathfrak{p}}|\bar{x}-\bar{y}|^{(\nu-1)(1+\mathfrak{p})-1}\Bigg).\label{est3.16}
\end{eqnarray}

Combining \eqref{est3.10}--\eqref{est3.16} with \eqref{est3.9}, we deduce
\begin{eqnarray}
0&\leq&2^{1-\mathfrak{p}}\mathrm{K}_{0}^{-\mathfrak{p}}(\mathfrak{L}_{1}+\|\tilde{f}\|_{L^{\infty}(Q_{1})}) \nonumber\\
&+&\mathfrak{L}_{2}|\bar{x}-\bar{y}|^{\nu-2}\Bigg(-8\min\{1,p-1\}\nu\left(\frac{1-\nu}{3-\nu}\right)\nonumber\\
&+&\frac{10^{-\mathfrak{p}}n\max\{1,p-1\}}{\mathfrak{L}_{2}|\bar{x}-\bar{y}|^{\nu-2}}\mathfrak{L}_{1}
+\frac{40n|p-2|}{\mathfrak{L}_{2}|\bar{x}-\bar{y}|^{\nu-1}}\mathfrak{L}_{1}\nonumber\\
&+&\frac{2^{4-2\mathfrak{p}}(1+|p-2|)\nu^{1+\mathfrak{p}}}{\mathfrak{L}_{2}^{-\mathfrak{p}}|\bar{x}-\bar{y}|^{(1-\nu)\mathfrak{p}}}\mathfrak{L}_{1}^{-\mathfrak{p}}\Bigg)\nonumber\\
&+&\mathfrak{a}(\bar{x},\bar{t})(2\mathrm{K}_{0})^{\mathfrak{q}-\mathfrak{p}}\mathfrak{L}_{2}|\bar{x}-\bar{y}|^{\nu-2}\Bigg(-8\min\{1,p-1\}\nu\left(\frac{1-\nu}{3-\nu}\right)\nonumber\\
&+&\frac{40n|p-2|}{\mathfrak{L}_{2}|\bar{x}-\bar{y}|^{\nu-1}}\mathfrak{L}_{1}\nonumber\\
&+&\frac{10^{-\mathfrak{p}}\!\left(1+\max\{10^{-\mathfrak{q}},10^{\mathfrak{q}}\}(2\mathrm{K}_{0})^{-\mathfrak{p}}\right)n\max\{1,p-1\}}{\mathfrak{L}_{2}|\bar{x}-\bar{y}|^{\nu-2}}\mathfrak{L}_{1}\nonumber\\
&+&\frac{2^{4-2\mathfrak{p}}(1+|p-2|)\nu^{1+\mathfrak{p}}|\bar{x}-\bar{y}|^{(\nu-1)\mathfrak{p}}}{\mathfrak{L}_{2}^{-\mathfrak{p}}}\mathfrak{L}_{1}^{-\mathfrak{p}}\nonumber\\
&+&\frac{4\mathfrak{a}(\bar{x},\bar{t})^{-1}\left(\frac{3}{40}\right)^{\mathfrak{p}}(2\mathrm{K}_{0})^{-\mathfrak{p}}n\max\{1,p-1\}[\mathfrak{a}]_{C^{0,1}_{x}(Q_{1})}\nu^{1+\mathfrak{p}}|\bar{x}-\bar{y}|^{(\nu-1)\mathfrak{p}+1}}{\mathfrak{L}_{2}^{-\mathfrak{p}}}\nonumber\\
&+&\frac{\mathfrak{a}(\bar{x},\bar{t})^{-1}\left(\frac{3}{40}\right)^{\mathfrak{p}}(2\mathrm{K}_{0})^{-\mathfrak{p}}n\max\{1,p-1\}[\mathfrak{a}]_{C^{0,1}_{x}(Q_{1})}\nu^{\mathfrak{p}}|\bar{x}-\bar{y}|^{(\nu-1)(\mathfrak{p}-1)+2}}{\mathfrak{L}_{2}^{1-\mathfrak{p}}}\mathfrak{L}_{1}\Bigg).\label{est3.17}
\end{eqnarray}

We now fix $\mathfrak{L}_{2}$ as
\begin{eqnarray*}
\mathfrak{L}_{2}&=&\Bigg(\frac{3}{\nu}+\frac{2^{1-\nu}40n|p-2|}{\min\{1,p-1\}\nu\left(\frac{1-\nu}{3-\nu}\right)}+\frac{2^{\frac{5-\nu-2\mathfrak{p}}{-\mathfrak{p}}}(1+|p-2|)^{-\frac{1}{\mathfrak{p}}}}{\min\{1,p-1\}^{-\frac{1}{\mathfrak{p}}}\nu\left(\frac{1-\nu}{3-\nu}\right)^{-\frac{1}{\mathfrak{p}}}}\\
&+&\frac{2^{2-\nu}\!\left(10^{-\mathfrak{p}}\!\left(1+\max\{10^{-\mathfrak{q}},10^{\mathfrak{q}}\}(2\mathrm{K}_{0})^{\mathfrak{q}-\mathfrak{p}}\right)n\max\{1,p-1\}+2^{1-\mathfrak{p}}\mathrm{K}_{0}^{-\mathfrak{p}}\right)}{\min\{1,p-1\}\nu\left(\frac{1-\nu}{3-\nu}\right)}\Bigg)\mathfrak{L}_{1}\\
&+&\Bigg(\frac{2^{3+(\nu-1)\mathfrak{p}}(2\mathrm{K}_{0})^{-\mathfrak{p}}(\mathfrak{a}_{-})^{-1}\left(\frac{3}{40}\right)^{\mathfrak{p}}n\max\{1,p-1\}\nu^{\mathfrak{p}}}{\min\{1,p-1\}\left(\frac{1-\nu}{3-\nu}\right)}\Bigg)^{-\frac{1}{\mathfrak{p}}}[\mathfrak{a}]_{C^{0,1}_{x}(Q_{1})}^{-\frac{1}{\mathfrak{p}}}\\
&+&\Bigg(\frac{2^{(\nu-1)(\mathfrak{p}-1)+2}(2\mathrm{K}_{0})^{-\mathfrak{p}}(\mathfrak{a}_{-})^{-1}\left(\frac{3}{40}\right)^{\mathfrak{p}}n\max\{1,p-1\}\nu^{\mathfrak{p}}}{\min\{1,p-1\}\nu\left(\frac{1-\nu}{3-\nu}\right)}\Bigg)^{\frac{1}{1-\mathfrak{p}}}(\mathfrak{L}_{1}[\mathfrak{a}]_{C^{0,1}_{x}(Q_{1})})^{\frac{1}{1-\mathfrak{p}}}\\
&+& \frac{2^{3-\nu-\mathfrak{p}}\mathrm{K}_{0}^{-\mathfrak{p}}}{\min\{1,p-1\}\nu\left(\frac{1-\nu}{3-\nu}\right)}\|\tilde{f}\|_{L^{\infty}(Q_{1})}\\
&\leq&\mathrm{C}\big(\|w\|_{L^{\infty}(Q_{1})}+ (\|w\|_{L^{\infty}(Q_{1})}[\mathfrak{a}]_{C^{0,1}_{x}(Q_{1})})^{\frac{1}{1-\mathfrak{p}}}+ [\mathfrak{a}]_{C^{0,1}_{x}(Q_{1})}^{-\frac{1}{\mathfrak{p}}}+\|\tilde{f}\|_{L^{\infty}(Q_{1})}\big).
\end{eqnarray*}

With this choice and the bounds for $\mathcal{I}_{1}$--$\mathcal{I}_{6}$, \eqref{est3.17} yields
\begin{eqnarray*}
0\leq -3(1+\mathfrak{a}_{-}(2\mathrm{K}_{0})^{\mathfrak{q}-\mathfrak{p}})\min\{1,p-1\}\nu\left(\frac{1-\nu}{3-\nu}\right)\mathfrak{L}_{2}|\bar{x}-\bar{y}|^{\nu-2}<0,
\end{eqnarray*}
a contradiction. Hence, by the definition of $\phi$ and $\phi\leq 0$,
\begin{equation*}
|w(x,t) - w(y,t)| \leq \mathfrak{L}_{2} |\bar{x} - \bar{y}|^{\nu}.
\end{equation*}
Moreover, $\mathfrak{L}_{2}\leq 1$ provided $\eta_{1}>0$ is chosen such that
\[
\|w\|_{L^{\infty}(Q_{1})} \leq \eta_{1}, 
\quad [\mathfrak{a}]_{C^{0,1}_{x}(Q_{1})} \leq \eta_{1}, 
\quad \text{and} \quad 
\|\tilde{f}\|_{L^{\infty}(Q_{1})} \leq \eta_{1}.
\]
\end{proof}

Now, we are in a position to prove the main result in this part.

\begin{proof}[\bf Proof of the Lemma \ref{LocalLiplemaprobtransl}]
We fix $x_{0},y_{0}\in B_{\frac{3}{4}}$ and $t_{0}\in \left(-\left(\frac{3}{4}\right)^{2},0\right)$. Define, for positive constants $\mathfrak{L}_{1}$ and $\mathfrak{L}_{2}$,
\begin{eqnarray*}
\phi(x,y,t)\coloneqq w(x,t)-w(y,t)-\mathfrak{L}_{2}\omega(|x-y|)-\frac{\mathfrak{L}_{1}}{2}|x-x_{0}|^2-\frac{\mathfrak{L}_{1}}{2}|y-y_{0}|^2-\frac{\mathfrak{L}_{1}}{2}(t-t_{0})^2.
\end{eqnarray*}
Here $\omega:[0,\infty)\to[0,\infty)$ is given by
\begin{eqnarray*}
\omega(s)=\begin{cases}
s-k_{0}s^{\gamma},& 0\leq s\leq s_{1}=\left(\frac{1}{\gamma k_{0}}\right)^{\frac{1}{\gamma-1}},\\
\omega(s_{1}),& \text{otherwise},
\end{cases}
\end{eqnarray*}
where
\[
\gamma=1-\frac{\mathfrak{p}\nu}{2}\in(1,2),
\]
with $\nu$ the H\"older exponent from Lemma~\ref{Localholderlemaprobtransl}, and $k_{0}>0$ is chosen so that $s_{1}>2$ and $\gamma k_{0}s_{1}^{\gamma-1}\leq \frac{1}{4}$. Then
\[
\omega'(s)=1-\gamma k_{0}s^{\gamma-1}, 
\qquad 
\omega''(s)=-\gamma(\gamma-1)k_{0}s^{\gamma-2}.
\]
By construction, $\omega'(s)\in[3/4,1]$ and $\omega''(s)<0$ for $s\in(0,2]$.

We claim that $\phi\leq 0$ in $\overline{B_{\frac{3}{4}}}\times \overline{B_{\frac{3}{4}}}\times\left[-\left(\frac{3}{4}\right)^{2},0\right]$. Arguing by contradiction, suppose $\phi(\bar{x},\bar{y},\bar{t})>0$ at a maximum point. Then $\bar{x}\neq \bar{y}$ and
\[
(\bar{x},\bar{y},\bar{t})\in B_{\frac{3}{4}}\times B_{\frac{3}{4}}\times\left(-\left(\frac{3}{4}\right)^{2},0\right),
\]
with
\[
\max\{|\bar{x}-x_{0}|,|\bar{y}-y_{0}|,|\bar{t}-t_{0}|\}\leq \frac{1}{16},
\]
upon choosing $\mathfrak{L}_{1}=4^{5}\|w\|_{L^{\infty}(Q_{1})}$. Using the $\nu$-H\"older continuity of $w$ and $\|w\|_{L^{\infty}(Q_{1})}\leq \eta_{0}\leq \eta_{1}$, we infer
\begin{equation*}
\max\{\mathfrak{L}_{1}|\bar{x}-x_{0}|^{2},\,\mathfrak{L}_{1}|\bar{y}-y_{0}|^{2}\}\leq 2|\bar{x}-\bar{y}|^{\nu}.
\end{equation*}

Applying the Jensen--Ishii lemma (Lemma~\ref{JensenIshii}), there exist
\[
(\tau+\mathfrak{L}_{1}(\bar{t}-t_{0}),\zeta_{1},\mathrm{X}+\mathfrak{L}_{1}\mathrm{Id}_{n})\in \overline{\mathcal{J}}^{+}(w)(\bar{x},\bar{t}),
\]
and
\[
(\tau,\zeta_{2},\mathrm{Y}-\mathfrak{L}_{1}\mathrm{Id}_{n})\in \overline{\mathcal{J}}^{-}(w)(\bar{y},\bar{t}),
\]
where
\[
\zeta_{1}=\mathfrak{L}_{2}\omega'(|\bar{x}-\bar{y}|)\frac{\bar{x}-\bar{y}}{|\bar{x}-\bar{y}|}+\mathfrak{L}_{1}(\bar{x}-x_{0}), 
\quad 
\zeta_{2}=\mathfrak{L}_{2}\omega'(|\bar{x}-\bar{y}|)\frac{\bar{x}-\bar{y}}{|\bar{x}-\bar{y}|}-\mathfrak{L}_{1}(\bar{y}-y_{0}).
\]
Assuming $\mathfrak{L}_{2}\geq \mathfrak{L}_{1}$ and using $\omega'\geq \frac{3}{4}$, we obtain
\begin{equation}\label{est3.18}
\frac{\mathfrak{L}_{2}}{4}\leq |\zeta_{i}|\leq 2\mathfrak{L}_{2}, \quad i=1,2.
\end{equation}

Moreover, the matrices $\mathrm{X}\leq \mathrm{Y}$ satisfy, for any $\kappa>0$,
\begin{equation*}
-[\kappa+2\|\mathrm{Z}\|]
\begin{pmatrix} \mathrm{Id}_n & 0 \\ 0 & \mathrm{Id}_n \end{pmatrix}
\leq 
\begin{pmatrix} \mathrm{X} & 0 \\ 0 & -\mathrm{Y} \end{pmatrix}
\leq 
\begin{pmatrix} \mathrm{Z} & -\mathrm{Z} \\ -\mathrm{Z} & \mathrm{Z} \end{pmatrix}
+\frac{2}{\kappa}
\begin{pmatrix} \mathrm{Z}^2 & -\mathrm{Z}^2 \\ -\mathrm{Z}^2 & \mathrm{Z}^2 \end{pmatrix},
\end{equation*}
where
\begin{align*}
\mathrm{Z} &= \mathfrak{L}_{2} \omega''(|\bar{x}-\bar{y}|)\,\xi\otimes\xi
+ \mathfrak{L}_{2} \omega'(|\bar{x}-\bar{y}|)\big(\mathrm{Id}_n-\xi\otimes\xi\big),\\
\mathrm{Z}^2 &=\left[ \mathfrak{L}_{2}^{2}\frac{(\omega'(|\bar{x}-\bar{y}|))^{2}}{|\bar{x}-\bar{y}|^{2}}+\mathfrak{L}_{2}^{2}(\omega''(|\bar{x}-\bar{y}|))^{2}\right]\big(\mathrm{Id}_n-\xi\otimes\xi\big),
\end{align*}
with $\xi=\frac{\bar{x}-\bar{y}}{|\bar{x}-\bar{y}|}$. Choosing
\[
\kappa=4\mathfrak{L}_{2}\left(|\omega''(|\bar{x}-\bar{y}|)|+\frac{|\omega'(|\bar{x}-\bar{y}|)|}{|\bar{x}-\bar{y}|}\right),
\]
we deduce
\[
\langle(\mathrm{X}-\mathrm{Y})\xi,\xi\rangle
\leq 2\mathfrak{L}_{2}\omega''(|\bar{x}-\bar{y}|)<0,
\]
and $\max\{\|\mathrm{X}\|,\|\mathrm{Y}\|\}\leq \kappa+2\|\mathrm{Z}\|$.

Finally, setting $\xi_{i}=\zeta_{i}+\vec{q}$, $i=1,2$, we obtain
\begin{equation}\label{est3.19}
\frac{\mathfrak{L}_{2}}{4}\leq |\xi_{i}|\leq 2|\vec{q}|, 
\qquad 
|\xi_{i}|\geq \frac{15}{16\mathrm{K}_{0}}|\vec{q}|,\quad i=1,2,
\end{equation}
where we used $\mathfrak{L}_{2}\leq 1$ (valid when $\max\{\|w\|_{L^{\infty}(Q_{1})},[\mathfrak{a}]_{C^{0,1}_{x}(Q_{1})},\|\tilde{f}\|_{L^{\infty}(Q_{1})}\}\leq \eta_{1}$) together with \eqref{est3.18}.

Now, invoking the viscosity inequalities, we obtain
\begin{eqnarray}
0&\leq& \frac{2(\mathfrak{L}_{1}+\|\tilde{f}\|_{L^{\infty}(Q_{1})})}{|\xi_{1}|^{\mathfrak{p}}}+\mathfrak{L}_{1}\left[\operatorname{tr}(\mathcal{A}(\xi_{1}))+\frac{|\xi_{2}|^{\mathfrak{p}}\operatorname{tr}(\mathcal{A}(\xi_{2}))}{|\xi_{1}|^{\mathfrak{p}}}\right]\nonumber\\
&+&\operatorname{tr}(\mathcal{A}(\xi_{1})(\mathrm{X}-\mathrm{Y}))
+\operatorname{tr}((\mathcal{A}(\xi_{1})-\mathcal{A}(\xi_{2}))\mathrm{Y})+\frac{(|\xi_{1}|^{\mathfrak{p}}-|\xi_{2}|^{\mathfrak{p}})\operatorname{tr}(\mathcal{A}(\xi_{2})\mathrm{Y})}{|\xi_{1}|^{\mathfrak{p}}}
\nonumber\\
&+&(\mathfrak{a}(\bar{x},\bar{t})-\mathfrak{a}(\bar{y},\bar{t}))\frac{|\xi_{2}|^{\mathfrak{q}}}{|\xi_{1}|^{\mathfrak{p}}}(\operatorname{tr}(\mathcal{A}(\xi_{2})\mathrm{Y})-\mathfrak{L}_{1}\operatorname{tr}(\mathcal{A}(\xi_{2})))\nonumber\\
&+&\mathfrak{a}(\bar{x},\bar{t})\left[\frac{|\xi_{1}|^{\mathfrak{q}}(\operatorname{tr}(\mathcal{A}(\xi_{1})\mathrm{X})-\mathfrak{L}_{1}\operatorname{tr}\mathcal{A}(\xi_{1})))}{|\xi_{1}|^{\mathfrak{p}}}-\frac{|\xi_{2}|^{\mathfrak{q}}(\operatorname{tr}(\mathcal{A}(\xi_{2})\mathrm{Y})-\mathfrak{L}_{1}\operatorname{tr}(\mathcal{A}(\xi_{2})))}{|\xi_{1}|^{\mathfrak{p}}}\right]\nonumber\\
&=:& \frac{2(\mathfrak{L}_{1}+\|\tilde{f}\|_{L^{\infty}(Q_{1})})}{|\xi_{1}|^{\mathfrak{p}}}+\mathcal{I}_{1}+\mathcal{I}_{2}+\mathcal{I}_{3}+\mathcal{I}_{4}+\mathcal{I}_{5}+\mathcal{I}_{6}.\label{est3.20}
\end{eqnarray}

Arguing as in Lemma \ref{Localholderlemaprobtransl}, we have:
\begin{itemize}
\item $\operatorname{tr}(\mathcal{A}(\xi_{i}))\in (n\min\{1,p-1\},n\max\{1,p-1\})$, $i=1,2$;
\item $\operatorname{tr}(\mathcal{A}(\xi_{i})(\mathrm{X}-\mathrm{Y}))\leq 2\mathfrak{L}_{2}\min\{1,p-1\}\omega''(|\bar{x}-\bar{y}|)$, $i=1,2$;
\item $\max\{\|\mathrm{X}\|,\|\mathrm{Y}\|\}\leq 4\mathfrak{L}_{2}\left(|\omega''(|\bar{x}-\bar{y}|)|+\frac{\omega'(|\bar{x}-\bar{y}|)}{|\bar{x}-\bar{y}|}\right)$;
\item $\frac{1}{3\mathrm{K}_{0}}\leq \frac{|\xi_{2}|}{|\xi_{1}|}\leq 3\mathrm{K}_{0}$ (by \eqref{est3.19}).
\end{itemize}

Consequently, we estimate
\begin{eqnarray*}
\mathcal{I}_{1}&\leq& 3\mathrm{K}_{0}n\max\{1,p-1\}\mathfrak{L}_{1},\\
\mathcal{I}_{2}&\leq& -2\mathfrak{L}_{2}\min\{1,p-1\}\gamma(\gamma-1)k_{0}|\bar{x}-\bar{y}|^{\gamma-2},\\
\mathcal{I}_{3}&\leq& 128n|p-2|\mathfrak{L}_{1}^{\frac{1}{2}}|\bar{x}-\bar{y}|^{\frac{\nu}{2}}(|\bar{x}-\bar{y}|^{-1}+\gamma(\gamma-1)k_{0}|\bar{x}-\bar{y}|^{\gamma-2}),\\
\mathcal{I}_{4}&\leq& 16n(1+|p-2|)\mathfrak{L}_{1}^{-\frac{\mathfrak{p}}{2}}\mathfrak{L}_{2}^{1+\mathfrak{p}}(|\bar{x}-\bar{y}|^{-1}+\gamma(\gamma-1)k_{0}|\bar{x}-\bar{y}|^{\gamma-2})|\bar{x}-\bar{y}|^{-\frac{\mathfrak{p}\nu}{2}},\\
\mathcal{I}_{5}&\leq&(3\mathrm{K}_{0})^{\mathfrak{q}-3\mathfrak{p}}n\max\{1,p-1\}[\mathfrak{a}]_{C^{0,1}_{x}(Q_{1})}(4^{1-\mathfrak{p}}\mathfrak{L}_{2}^{1+\mathfrak{p}}(1+\gamma(\gamma-1)k_{0}|\bar{x}-\bar{y}|^{\gamma-1})\\
&+&4^{-\mathfrak{p}}\mathfrak{L}_{1}\mathfrak{L}_{2}^{\mathfrak{p}}),\\
\mathcal{I}_{6}&\leq&\mathfrak{a}(\bar{x},\bar{t})(2\mathrm{K}_{0})^{\mathfrak{q}-\mathfrak{p}}\Big(2\mathfrak{L}_{2}\min\{1,p-1\}\gamma(\gamma-1)k_{0}|\bar{x}-\bar{y}|^{\gamma-2}\\
&+&4^{-\mathfrak{p}}(1+\max\{(3\mathrm{K}_{0})^{-\mathfrak{q}},(3\mathrm{K}_{0})^{\mathfrak{q}}\}(2\mathrm{K}_{0})^{-\mathfrak{p}})n\max\{1,p-1\}\mathfrak{L}_{2}^{\mathfrak{p}}\mathfrak{L}_{1}\\
&+& 128n|p-2|\mathfrak{L}_{1}^{\frac{1}{2}}|\bar{x}-\bar{y}|^{\frac{\nu}{2}}(|\bar{x}-\bar{y}|^{-1}+\gamma(\gamma-1)k_{0}|\bar{x}-\bar{y}|^{\gamma-2})\\
&+&\max\{2^{2+\mathfrak{q}},2^{2-2\mathfrak{q}}\}n(1+|p-2|)\mathfrak{L}_{1}^{-\frac{\mathfrak{p}}{2}}\mathfrak{L}_{2}^{1+\mathfrak{p}}(|\bar{x}-\bar{y}|^{-1}\\
&+&\gamma(\gamma-1)k_{0}|\bar{x}-\bar{y}|^{\gamma-2})|\bar{x}-\bar{y}|^{-\frac{\mathfrak{p}\nu}{2}}\Big).
\end{eqnarray*}

Substituting into \eqref{est3.20} yields
\begin{eqnarray} 0&\leq& 2^{1-\mathfrak{p}}\mathrm{K}_{0}^{-\mathfrak{p}}(\mathfrak{L}_{1}+\|\tilde{f}\|_{L^{\infty}(Q_{1})})+\mathfrak{L}_{2}|\bar{x}-\bar{y}|^{\gamma-2}\Bigg(-2\min\{1,p-1\}\gamma(\gamma-1)k_{0}\nonumber\\ &+&\frac{3\mathrm{K}_{0}n\max\{1,p-1\}}{\mathfrak{L}_{2}|\bar{x}-\bar{y}|^{\gamma-2}}\mathfrak{L}_{1}+\frac{128n|p-2||\bar{x}-\bar{y}|^{1-\gamma+\frac{\nu}{2}}(1+\gamma(\gamma-1)k_{0}|\bar{x}-\bar{y}|^{\gamma-1})}{\mathfrak{L}_{2}}\mathfrak{L}_{1}^{\frac{1}{2}}\nonumber\\ &+&\frac{16n(1+|p-2|)(1+\gamma(\gamma-1)k_{0}|\bar{x}-\bar{y}|^{\gamma-1})}{\mathfrak{L}_{2}^{-\mathfrak{p}}}\mathfrak{L}_{1}^{-\frac{\mathfrak{p}}{2}}\Bigg)\nonumber\\ &+&\mathfrak{a}(\bar{x},\bar{t})(2\mathrm{K}_{0})^{\mathfrak{q}-\mathfrak{p}}\mathfrak{L}_{2}|\bar{x}-\bar{y}|^{\gamma-2}\Bigg(-2\min\{1,p-1\}\gamma(\gamma-1)k_{0}\nonumber\\ &+&\frac{4^{-\mathfrak{p}}(1+\max\{(3\mathrm{K}_{0})^{-\mathfrak{q}},(3\mathrm{K}_{0})^{\mathfrak{q}}\}(2\mathrm{K}_{0})^{-\mathfrak{p}})n\max\{1,p-1\}}{\mathfrak{L}_{2}^{1-\mathfrak{p}}|\bar{x}-\bar{y}|^{\gamma-2}}\mathfrak{L}_{1}\nonumber\\ &+&\frac{128n|p-2||\bar{x}-\bar{y}|^{1-\gamma-\frac{\nu}{2}}(1+\gamma(\gamma-1)k_{0}|\bar{x}-\bar{y}|^{\gamma-1})}{\mathfrak{L}_{2}}\mathfrak{L}_{1}^{\frac{1}{2}}\nonumber\\ &+&\frac{\max\{2^{2+\mathfrak{q}},2^{2-2\mathfrak{q}}\}n(1+|p-2|)(1+\gamma(\gamma-1)k_{0}|\bar{x}-\bar{y}|^{\gamma-1})}{\mathfrak{L}_{2}^{-\mathfrak{p}}}\mathfrak{L}_{1}^{-\frac{\mathfrak{p}}{2}}\nonumber\\ &+&\frac{2^{2-\mathfrak{p}+\mathfrak{q}}3^{\mathfrak{q}-3\mathfrak{p}}\mathfrak{a}(\bar{x},\bar{t})^{-1}\mathrm{K}_{0}^{-2\mathfrak{p}}[\mathfrak{a}]_{C^{0,1}_{x}(Q_{1})}n\max\{1,p-1\}(1+\gamma(\gamma-1)k_{0}|\bar{x}-\bar{y}|^{\gamma-1})}{\mathfrak{L}_{2}^{-\mathfrak{p}}|\bar{x}-\bar{y}|^{\gamma-2}}\nonumber\\ &+&\frac{2^{-2\mathfrak{p}}\mathfrak{a}(\bar{x},\bar{t})^{-1}(3\mathrm{K}_{0})^{\mathfrak{q}-3\mathfrak{p}}(2\mathrm{K}_{0})^{\mathfrak{p}-\mathfrak{q}}[\mathfrak{a}]_{C^{0,1}_{x}(Q_{1})}n\max\{1,p-1\}}{\mathfrak{L}_{2}^{1-\mathfrak{p}}|\bar{x}-\bar{y}|^{\gamma-2}}\mathfrak{L}_{1}\Bigg).\label{est3.21} \end{eqnarray}

We now fix $\mathfrak{L}_{2}$ as follows:
\begin{eqnarray*} \mathfrak{L}_{2}&=&\Bigg(1+\left(\frac{2^{2-\gamma-2\mathfrak{p}}5(1+\max\{(3\mathrm{K}_{0})^{-\mathfrak{q}},(3\mathrm{K}_{0})^{\mathfrak{q}}\}(2\mathrm{K}_{0})^{-\mathfrak{p}})n\max\{1,p-1\}}{\min\{1,p-1\}\gamma(\gamma-1)k_{0}}\right)^{\frac{1}{1-\mathfrak{p}}}\\ &+&\frac{2^{2-\gamma}(15\mathrm{K}_{0}n\max\{1,p-1\}+2^{1-\mathfrak{p}}5\mathrm{K}_{0}^{-\mathfrak{p}})}{\min\{1,p-1\}\gamma(\gamma-1)k_{0}}\Bigg)\mathfrak{L}_{1}+\frac{2^{3-\gamma-\mathfrak{p}}\mathrm{K}_{0}^{-\mathfrak{p}}}{\min\{1,p-1\}\gamma(\gamma-1)k_{0}}\|\tilde{f}\|_{L^{\infty}(Q_{1})}\nonumber\\ &+&\Bigg(\left(\frac{20(4+\max\{2^{\mathfrak{q}},2^{-2\mathfrak{q}}\})n(1+|p-2|)(1+2^{\gamma-1}\gamma(\gamma-1)k_{0})}{\min\{1,p-1\}\gamma(\gamma-1)k_{0}}\right)^{-\frac{1}{\mathfrak{p}}}\nonumber\\ &+&\frac{640n|p-2|(1+2^{\gamma-1}\gamma(\gamma-1)k_{0})}{\min\{1,p-1\}\gamma(\gamma-1)k_{0}}\Bigg)\mathfrak{L}_{1}^{\frac{1}{2}}\\ &+&\Bigg(\frac{2^{4-\gamma-\mathfrak{p}+\mathfrak{q}}3^{\mathfrak{q}-3\mathfrak{p}}5\mathrm{K}_{0}^{-2\mathfrak{p}}(\mathfrak{a}_{-})^{-1}n\max\{1,p-1\}(1+2^{\gamma-1}\gamma(\gamma-1)k_{0})}{\min\{1,p-1\}\gamma(\gamma-1)k_{0}}\Bigg)^{-\frac{1}{\mathfrak{p}}}[\mathfrak{a}]_{C^{0,1}_{x}(Q_{1})}^{-\frac{1}{\mathfrak{p}}}\\ &+&\Bigg(\frac{2^{2-\gamma-\mathfrak{p}}3^{\mathfrak{q}-3\mathfrak{p}}5\mathrm{K}_{0}^{-2\mathfrak{p}}(\mathfrak{a}_{-})^{-1}n\max\{1,p-1\}}{\min\{1,p-1\}\gamma(\gamma-1)k_{0}}\Bigg)^{\frac{1}{1-\mathfrak{p}}}(\mathfrak{L}_{1}[\mathfrak{a}]_{C^{0,1}_{x}(Q_{1})})^{\frac{1}{1-\mathfrak{p}}}\\ &\leq&\mathrm{C}(\|w\|_{L^{\infty}(Q_{1})}+\|w\|_{L^{\infty}(Q_{1})}^{\frac{1}{2}}+ (\|w\|_{L^{\infty}(Q_{1})}[\mathfrak{a}]_{C^{0,1}_{x}(Q_{1})})^{\frac{1}{1-\mathfrak{p}}}+ [\mathfrak{a}]_{C^{0,1}_{x}(Q_{1})}^{-\frac{1}{\mathfrak{p}}}+\|\tilde{f}\|_{L^{\infty}(Q_{1})}), \end{eqnarray*}

With this choice and the bounds on $\mathcal{I}_{1}$--$\mathcal{I}_{6}$, \eqref{est3.21} yields
\begin{equation}
0\leq-(1+\mathfrak{a}_{-}(2\mathrm{K}_{0})^{\mathfrak{q}-\mathfrak{p}})\min\{1,p-1\}\gamma(\gamma-1)k_{0}\mathfrak{L}_{2}|\bar{x}-\bar{y}|^{\gamma-2}<0,
\end{equation}
a contradiction.

As in the H\"{o}lder estimates, there exists $\eta_{1}>0$ such that if
\[
\|w\|_{L^{\infty}(Q_{1})}\leq \eta_{1},\quad [\mathfrak{a}]_{C^{0,1}_{x}(Q_{1})}\leq \eta_{1}, \quad \|\tilde{f}\|_{L^{\infty}(Q_{1})}\leq \eta_{1},
\]
then $\mathfrak{L}_{2}\leq 1$, and consequently
\begin{eqnarray*}
|w(x,t)-w(y,t)|\leq |x-y|,\quad \forall (x,t),(y,t)\in Q_{\frac{3}{4}}.
\end{eqnarray*}
\end{proof}

\section{
Proof of Theorem \ref{Thm01}}\label{Sec4}

Based on the compactness established in the previous section, we proceed with the second and third steps in the proof of Theorem \ref{Thm01}. We first state an approximation lemma, whose proof follows that of \cite[Lemma 5.1]{BessaDaSilvaSa25} and is therefore omitted.

\begin{lemma}[\bf Approximation Scheme]\label{Approximationlemma}
Assume that the structural conditions $\mathbf{(A1)}-\mathbf{(A3)}$ hold. Let \(u\) be a viscosity solution to \eqref{Problem} such that \(\displaystyle\operatornamewithlimits{osc}_{Q_1} u \leq 5\) and $u(0,0)=0$. Given \(\delta > 0\), there exists \(\eta \in (0,1)\), depending only on \(p\), \(n\), \(\mathfrak{p}\), \(\mathfrak{q}\), \(\mathfrak{a}_{-}\), and \(\delta\), such that if
\[
\max\left\{[\mathfrak{a}]_{C^{0,1}_{x}(Q_{1})},\|f\|_{L^{\infty}(Q_{1})}\right\}\leq \eta,
\]
then there exists a solution \(\mathfrak{h}\) to
\begin{eqnarray}\label{probhom}
\partial_{t}\mathfrak{h}-(|D\mathfrak{h}|^{\mathfrak{p}}+\mathfrak{a}(0,t)|D\mathfrak{h}|^{\mathfrak{q}})\Delta_{p}^{\mathrm{N}}\mathfrak{h}=0\,\,\, \text{in}\,\,\, Q_{7/8}
\end{eqnarray}
satisfying $\mathfrak{h}(0,0)=0$ and
\[
\|u-\mathfrak{h}\|_{L^{\infty}(Q_{7/8})}\leq\delta.
\]
\end{lemma}

This approximation result yields, under suitable smallness conditions, an improvement of flatness for solutions to \eqref{Problem}, as stated below.

\begin{lemma}[\textbf{Improvement of flatness}]\label{improvflat1}
Assume that the structural conditions $\mathbf{(A1)}-\mathbf{(A3)}$ hold. Let $\eta_1$ be the constant from Lemma \ref{LocalLiplemaprobtransl}. Let \(u\) be a viscosity solution to \eqref{Problem} such that \(\displaystyle\operatornamewithlimits{osc}_{Q_1} u \leq 5\) and $u(0,0)=0$. There exist constants \(\eta > 0\), \(\rho > 0\), and \(\theta \in (0, 1)\), depending only on \(p\), \(n\), \(\mathfrak{p}\), \(\mathfrak{q}\), and $\mathfrak{A}_0$, with \(\rho < (1 - \theta)^{1 + \mathfrak{p}}\), such that if
\[
\max\left\{[\mathfrak{a}]_{C^{0,1}_{x}(Q_{1})},\|f\|_{L^{\infty}(Q_{1})}\right\}\leq \eta,
\]
then there exists a universally bounded vector \(\mathfrak{v}\) such that
\[
\operatornamewithlimits{osc}_{(x,t)\in Q_{\rho}^{(1-\theta)}}(u(x,t)-\mathfrak{v}\cdot x)\leq \eta_1\rho(1-\theta).
\]
\end{lemma}

\begin{proof}
Fix \(\delta > 0\), to be specified \textit{a posteriori}. By Lemma \ref{Approximationlemma}, there exist a universal constant \(\eta > 0\) and a function \(\mathfrak{h}\), a viscosity solution to
\[
\partial_{t}\mathfrak{h} - (|D\mathfrak{h}|^{\mathfrak{p}}+\mathfrak{a}(0,t)|D\mathfrak{h}|^{\mathfrak{q}})\Delta_{p}^{\mathrm{N}}\mathfrak{h} = 0 \,\,\, \text{in} \,\,\, Q_{7/8},
\]
such that $\mathfrak{h}(0,0)=0$ and
\begin{eqnarray}\label{ineq0prop5.2}
\|u - \mathfrak{h}\|_{L^{\infty}(Q_{7/8})} \leq \delta.
\end{eqnarray}
Thus, \(\eta\) is determined by the choice of \(\delta\). By the regularity result in \cite[Theorem 2.3]{FZ23}, there exist universal constants \(\mathrm{C}' > 0\) and \(\alpha_{1} \in (0,1)\) such that, for every \(r \in (0, 5/8)\), there exists \(\mathfrak{v} \in \mathbb{R}^{n}\) satisfying
\[
|\mathfrak{v}| \leq \mathrm{C^{\prime}}(n, p, \mathfrak{p}, \mathfrak{q}, \mathfrak{a}_{-}, \mathfrak{a}_{+}, \mathfrak{A}_{0}, \|\mathfrak{h}\|_{L^{\infty}(Q_{7/8})}),
\]
and
\[
\operatornamewithlimits{osc}_{(x,t) \in Q_{r}}(\mathfrak{h}(x, t) - \mathfrak{v} \cdot x) \leq \mathrm{C^{\prime}}r^{1 + \alpha_{1}}.
\]

Choose \(r_{0} \in (0, 5/8)\) such that
\begin{eqnarray}\label{ineq1prop5.2}
\operatornamewithlimits{osc}_{(x,t) \in Q_{r_{0}}}(\mathfrak{h}(x, t) - \mathfrak{v} \cdot x) \leq \frac{\eta_1}{2}r_{0}(1-\theta)^{2+\mathfrak{p}},
\end{eqnarray}
for some \(\theta \in (0,1)\). Set
\begin{eqnarray}
\rho = r_{0}(1-\theta)^{1+\mathfrak{p}}, \qquad \delta = \frac{\eta_1}{4}\rho(1-\theta).
\end{eqnarray}
By the definition of \(\rho\), \eqref{ineq1prop5.2} implies
\begin{eqnarray}\label{ineq2prop5.2}
\operatornamewithlimits{osc}_{(x,t) \in Q_{\rho}^{(1-\theta)}}(\mathfrak{h}(x, t) - \mathfrak{v} \cdot x) \leq \frac{\eta_1}{2}\rho(1-\theta).
\end{eqnarray}
Since \(r_{0} < 1\), it follows that \(\rho < (1-\theta)^{1+\mathfrak{p}}\) and \(\rho < 7/8\). Combining the choice of \(\delta\) with \eqref{ineq0prop5.2} and \eqref{ineq2prop5.2}, we obtain
\begin{align*}
\operatornamewithlimits{osc}_{(x,t) \in Q_{\rho}^{(1-\theta)}}(u(x, t) - \mathfrak{v} \cdot x)
&\leq \operatornamewithlimits{osc}_{(x,t) \in Q_{\rho}^{(1-\theta)}}(u(x, t) - \mathfrak{h}(x, t)) 
+ \operatornamewithlimits{osc}_{(x,t) \in Q_{\rho}^{(1-\theta)}}(\mathfrak{h}(x, t) - \mathfrak{v} \cdot x)\\
&\leq 2\|u - \mathfrak{h}\|_{L^{\infty}(Q_{7/8})} + \frac{\eta_1}{2}\rho(1-\theta)\\
&\leq 2\delta + \frac{\eta_1}{2}\rho(1-\theta)\\
&\leq \eta_1\rho(1-\theta).
\end{align*}
This completes the proof.
\end{proof}

With this lemma in place, we proceed with the associated iterative scheme.

\begin{lemma}[\bf Iterative argument]\label{improvflat2}
Assume that the structural conditions \(\mathbf{(A1)}\)-\(\mathbf{(A3)}\) hold. Let $\eta_1$ be the constant defined in Lemma \ref{LocalLiplemaprobtransl}. Let \(u\) be a viscosity solution to \eqref{Problem} satisfying \(\displaystyle\operatornamewithlimits{osc}_{Q_{1}} u \leq 1\) and $u(0,0)=0$. Moreover, assume that 
\[
\max\left\{[\mathfrak{a}]_{C^{0,1}_{x}(Q_{1})},\|f\|_{L^{\infty}(Q_{1})}\right\}\leq \eta,
\]
where \(\eta > 0\) is the constant from Lemma \ref{improvflat1}, with \(\eta<\eta_{1}\). Then there exist constants \(\rho \in (0, 1)\) and \(\theta \in (0, 1)\), with \(\rho < (1 - \theta)^{1 + \mathfrak{p}}\), such that the following holds for every integer \(j \geq 0\): if for each \(i \in \{0, 1, \dots, j\}\) there exists \(\ell_i \in \mathbb{R}^{n}\) such that
\begin{itemize}
\item[(i)] \(|\ell_{i}|\leq 2(1-\theta)^{i}\),
\item[(ii)] \(\displaystyle\operatornamewithlimits{osc}_{(x,t)\in Q_{\rho^{i}}^{(1-\theta)^{i}}}(u(x,t)-\ell_{i}\cdot x)\leq \eta_1\rho^{i}(1-\theta)^{i},\)
\end{itemize}
then there exists a vector \(\ell_{j+1} \in \mathbb{R}^{n}\) such that
\begin{eqnarray}\label{teseimprov2flat}
\displaystyle\operatornamewithlimits{osc}_{(x,t)\in Q_{\rho^{j+1}}^{(1-\theta)^{j+1}}}(u(x,t)-\ell_{j+1}\cdot x)\leq\eta_1 \rho^{j+1}(1-\theta)^{j+1}\,\,\, \text{and}\,\,\, |\ell_{j+1}-\ell_{j}|\leq \mathrm{C_{3}}(1-\theta)^{j},
\end{eqnarray}
for some universal constant \(\mathrm{C_{3}} > 0\).
\end{lemma}

\begin{proof}
Let \(\rho, \theta, \eta\), and \(\mathrm{C}'\) be the constants provided by Lemma \ref{improvflat1}, and define
\[
\mathrm{C}_{1} = 2+\mathrm{C}'.
\]
With this choice, we proceed to establish the claim. 
If \(j = 0\), set \(\ell_{0} = 0\). Then, by the assumption \(\operatornamewithlimits{osc}_{Q_{1}} u \leq 1 \), the conclusion follows directly from Lemma \ref{improvflat1}. Assume now that \(j > 0\), and suppose that the statement holds for \(i = 0, 1, \ldots, j\). We show that there exists a vector \(\ell_{j+1}\) such that \eqref{teseimprov2flat} holds. 

Define the auxiliary function
\[
u_{j}(x,t)=\frac{u(\rho^{j}x,\rho^{2j}(1-\theta)^{-j\mathfrak{p}}t)-\ell_{j}\cdot (\rho^{j}x)}{\rho^{j}(1-\theta)^{j}}.
\]
Note that, by hypothesis,
\[
\operatornamewithlimits{osc}_{Q_1} u_{j} \leq 5
\quad \text{and} \quad
|\ell_{j}| \leq 2(1-\theta)^{j}.
\]
Next, define
\[
v_{j}(x,t) = u_{j}(x,t) + \vec{q} \cdot x,
\quad \text{where} \quad
\vec{q} = \frac{\ell_{j}}{(1 - \theta)^{j}}.
\]
Then \(v_{j}\) solves
\[
\partial_{t}v_{j}-(|D u_j|^{\mathfrak{p}}+\mathfrak{a}_{j}(x,t)|D u_j|^{\mathfrak{q}})\Delta_{p}^{\mathrm{N}}v_{j}
= f_{j}(x,t)
\quad \text{in} \quad Q_{1}.
\] 
for
\begin{eqnarray*}
\left\{
\begin{array}{rll}
f_{j}(x,t)& \defeq& \frac{\rho^{j}}{(1-\theta)^{j(1+\mathfrak{p})}}f(\rho^{j}x,\rho^{2j}(1-\theta)^{-j\mathfrak{p}}t),\\
\mathfrak{a}_{j}(x,t)&\coloneqq &(1-\theta)^{j(\mathfrak{q}-\mathfrak{p})}\mathfrak{a}(\rho^{j}x,\rho^{2j}(1-\theta)^{-j\mathfrak{p}}t).
\end{array}
\right.
\end{eqnarray*}
Observe that \(\mathfrak{a}_{j}\in C^{1}_{x}(Q_{1})\cap C^{1}_{t}(Q_{1})\), and for $(x,t)\in Q_{1}$ we have
\begin{align*}
|D_{x,t}\mathfrak{a}_{j}(x,t)|^{2}&\leq\rho^{2j}|D \mathfrak{a}(\rho^{j}x,\rho^{2j}(1-\theta)^{-j\mathfrak{p}}t)|^{2}+\rho^{4j}(1-\theta)^{-2j\mathfrak{p}}|\partial_{t}\mathfrak{a}(\rho^{j}x,\rho^{2j}(1-\theta)^{-j\mathfrak{p}}t)|^{2}\\
&\leq|D \mathfrak{a}(\rho^{j}x,\rho^{2j}(1-\theta)^{-j\mathfrak{p}}t)|^{2}+(1-\theta)^{2j(2+\mathfrak{p})}|\partial_{t}\mathfrak{a}(\rho^{j}x,\rho^{2j}(1-\theta)^{-j\mathfrak{p}}t)|^{2}\\
&\leq|D_{x,t}\mathfrak{a}(\rho^{j}x,\rho^{2j}(1-\theta)^{-j\mathfrak{p}}t)|^{2}.
\end{align*}
Consequently,
\[
\|D_{x,t}\mathfrak{a}_{j}\|_{L^{\infty}(Q_{1})}\leq \|D_{x,t}\mathfrak{a}\|_{L^{\infty}(Q_{\rho^{j}}^{(1-\theta)^{j}})}\leq \mathfrak{A}_{0}<\infty.
\] 
Moreover, for $t\in(-1,0]$, we have
\begin{align*}
|\mathfrak{a}_{j}(x,t)-\mathfrak{a}_{j}(y,t)| &= (1-\theta)^{j(\mathfrak{q}-\mathfrak{p})}|\mathfrak{a}(\rho^{j}x,\rho^{2j}(1-\theta)^{-j\mathfrak{p}}t)-\mathfrak{a}(\rho^{j}x,\rho^{2j}(1-\theta)^{-j\mathfrak{p}}t)|\\
&\leq (1-\theta)^{j(\mathfrak{q}-\mathfrak{p})}\left(\rho^j|x-y|\right)\\
&\leq (1-\theta)^{j(\mathfrak{q}-\mathfrak{p})}(1-\theta)^{j+\mathfrak{p}j}\eta|x-y|\\
&= (1-\theta)^{j(\mathfrak{q}+1)}\eta|x-y|,
\end{align*}
that is, $[\mathfrak{a}]_{C^{0,1}_{x}(Q_{1})}\leq\eta$. Moreover,
\[
\operatornamewithlimits{osc}_{Q_{1}}v_{j}\leq 1+2|\vec{q}|\leq 5
\]
and
\[
\|f_{j}\|_{L_{\infty}(Q_{1})}=\left(\frac{\rho}{(1-\theta)^{1+\mathfrak{p}}}\right)^{j}\|f\|_{L^{\infty}(Q_{\rho^{j}}^{(1-\theta)^{j}})}\leq \|f\|_{L^{\infty}(Q_{1})}\leq \eta.
\]
Since \(\rho < (1 - \theta)^{1 + \mathfrak{p}}\), the assumptions of Lemma \ref{improvflat1} are satisfied. Consequently, there exists \(\mathfrak{v} \in \mathbb{R}^{n}\), with \(|\mathfrak{v}| \leq \mathrm{C}'\), such that
\begin{eqnarray}\label{est1improv2flat}
\operatornamewithlimits{osc}_{(x,t)\in Q_{\rho}^{(1-\theta)}}(v_{j}(x,t)-\mathfrak{v}\cdot x)\leq \eta_1\rho(1-\theta).
\end{eqnarray}
Define \(\ell_{j+1} = (1 - \theta)^{j} \mathfrak{v}\). Then
\[
|\ell_{j+1}-\ell_{j}|\leq (|\mathfrak{v}|+|\ell_{j}|)(1-\theta)^{j}\leq (\mathrm{C^{\prime}}+2)(1-\theta)^{j}=\mathrm{C_{1}}(1-\theta)^{j}.
\]
Finally, scaling \eqref{est1improv2flat} yields
\[
\operatornamewithlimits{osc}_{Q_{\rho^{j+1}}^{(1-\theta)^{j+1}}}(u(x,t)-\ell_{j+1}\cdot x)\leq \eta_1\rho^{j+1}(1-\theta)^{j+1},
\]
which establishes \eqref{teseimprov2flat}. 
\end{proof}

We are now in a position to establish the Hölder continuity of \(Du\) at the origin, as well as the improved Hölder regularity of \(u\) with respect to the time variable. These estimates suffice to prove Theorem \ref{Thm01} via standard scaling and normalization arguments. From this point on, we assume that \(\eta < \eta_1\); this ensures that the smallness condition required in Lemma \ref{Localholderlemaprobtransl} is satisfied upon entering the second alternative of the iteration.

\begin{proposition}
Assume that the structural conditions \(\mathbf{(A1)}\)-\(\mathbf{(A3)}\) hold. Let \(u\) be a viscosity solution to \eqref{Problem} satisfying \(\displaystyle\operatornamewithlimits{osc}_{Q_{1}} u \leq 1\). If
\[
\max\left\{[\mathfrak{a}]_{C^{0,1}_{x}(Q_{1})},\|f\|_{L^{\infty}(Q_{1})}\right\}\leq \eta,
\]
where \(\eta > 0\) is the constant from Lemma \ref{improvflat1}, then there exist \(\alpha \in \left(0, \frac{1}{1 - \mathfrak{p}}\right)\) and a constant \(\mathrm{C} > 0\), depending only on \(n\), \(p\), \(\mathfrak{p}\), \(\mathfrak{q}\), \(\mathfrak{a}_{-}\), \(\mathfrak{a}_{+}\), and \(\mathfrak{A}_{0}\), such that
\[
|Du(x,t)-D u(y,s)|\leq \mathrm{C}(|x-y|^{\alpha}+|t-s|^{\frac{\alpha}{2}})
\]
and
\[
|u(x,t)-u(x,s)|\leq \mathrm{C}|t-s|^{\frac{1+\alpha}{2}}.
\]
\end{proposition}

\begin{proof}
First, up to a translation, we may assume without loss of generality that $u(0,0)=0$. Let \(\theta\) and \(\rho\) be the constants from Lemma \ref{improvflat1}, and let \(j\) be the smallest integer for which conditions (i) and (ii) in that lemma fail. In this case, by Lemma \ref{improvflat2}, it follows that for any \(\mathfrak{v} \in \mathbb{R}^{n}\) with \(|\mathfrak{v}| \leq 2(1 - \theta)^{j}\), the following estimate holds:
\[
|u(x,t)-\mathfrak{v}\cdot x|\leq \mathrm{C}(|x|^{1+\zeta}+|t|^{\frac{1+\zeta}{2-\zeta\mathfrak{p}}}),\,\, \forall (x,t)\in Q_{1}\setminus Q_{\rho^{j+1}}^{(1-\theta)^{j+1}},
\]
where \(\zeta = \frac{\log(1 - \theta)}{\log \rho}\) and \(\mathrm{C} = \frac{3 + \mathrm{C_{1}}(1 - \theta)^{-1}}{\rho(1 - \theta)}\). We distinguish two cases according to the value of \(j\):\\
\(\textbf{1°)}\) \(\bf{j=\infty}\)\\
In this case, the desired regularity follows with 
\(\alpha=\min\{1,\zeta\}\in \left(0,\frac{1}{1-\mathfrak{p}}\right)\), since the relation between $\theta$ and $\rho$ in Lemma \ref{improvflat1} yields
\[
0<\alpha\leq \zeta <1+\mathfrak{p}\leq \frac{1}{1-\mathfrak{p}}.
\]
More precisely, for each \(i\in\mathbb{N}\), there exists a vector \(\ell_{i}\in\mathbb{R}^{n}\) with \(|\ell_{i}|\leq 2(1-\theta)^{i}\) such that
\[
\operatornamewithlimits{osc}_{(x,t)\in Q_{\rho^{i}}^{(1-\theta)^{i}}}(u(x,t)-\ell_{i}\cdot x)\leq \eta_1\rho^{i}(1-\theta)^{i},
\]
and the regularity follows from the equivalence between Campanato and H\"{o}lder spaces (see \cite[Lemma 4.3]{Lieberman96}).\\
\(\textbf{2°)}\) \(\bf{j<\infty}\)\\
Applying Lemma \ref{improvflat2}, we conclude that for all \(i = 0, 1, \ldots, j\), there exists a vector \(\ell_{i} \in \mathbb{R}^{n}\) such that
\begin{equation}\label{Vasco}
\operatornamewithlimits{osc}_{(x,t)\in Q_{\rho^{i}}^{(1-\theta)^{i}}}(u(x,t)-\ell_{i}\cdot x)\leq \eta_1\rho^{i}(1-\theta)^{i}.
\end{equation}
Moreover, 
\[|\ell_{i}|\leq 2(1-\theta)^{i}\quad\text{for}\quad i=0,1,\ldots, j-1, \quad \text{and}\quad |\ell_{j}-\ell_{j-1}|\leq \mathrm{C}_{1}(1-\theta)^{j-1}.
\]
By the minimality of \(j\), we have \(|\ell_{j}|\geq 2(1-\theta)^{j}\). This implies
\begin{eqnarray}\label{ineq1Teo5.4}
2(1-\theta)^{j}\leq |\ell_{j}|\leq |\ell_{j}-\ell_{j-1}|+|\ell_{j-1}|\leq (\mathrm{C_{1}}+2)(1-\theta)^{j-1}.
\end{eqnarray}
Define the auxiliary function
\[
u_{j}(x,t)=\frac{u(\rho^{j}x,\rho^{2j}(1-\theta)^{-j\mathfrak{p}}t)-\ell_{j}\cdot (\rho^{j}x)}{\rho^{j}(1-\theta)^{j}}.
\]
From \eqref{Vasco}, we have
\[\displaystyle\operatornamewithlimits{osc}_{Q_1} u_{j} \leq \eta_1.\]
Moreover, 
\(u_{j}\) is a viscosity solution to
\[
\partial_{t}u_{j}-\mathcal{H}_{j}(x,t,Du_{j})\Delta_p^\mathrm{N}( u_{j} +\vec{q}\cdot x)=f_{j}(x,t)\,\,\, \text{in}\,\,\, Q_{1},
\]
where
\begin{equation}\label{aux_eqHolder}
\left\{
\begin{array}{rcl}
\vec{q}&\defeq& \frac{\ell_{j}}{(1-\theta)^{j}},\\
f_{j}(x,t)& \defeq& \frac{\rho^{j}}{(1-\theta)^{j(1+\mathfrak{p})}}f(\rho^{j}x,\rho^{2j}(1-\theta)^{-j\mathfrak{p}}t), \\
\mathcal{H}_{j}(x,t,\vec{\xi})&\defeq&|\vec{\xi}+\vec{q}|^\mathfrak{p}+\mathfrak{a}_{j}(x,t)|\vec{\xi}+\vec{q}|^\mathfrak{q},\\
\mathfrak{a}_{j}(x,t)&\defeq&(1-\theta)^{j(\mathfrak{q}-\mathfrak{p})}\mathfrak{a}(\rho^{j}x,\rho^{2j}(1-\theta)^{-j\mathfrak{p}}t).
\end{array}
\right.
\end{equation}
As shown in the proof of Lemma \ref{improvflat2}, we have 
\[
\max\left\{[\mathfrak{a}_j]_{C^{0,1}_{x}(Q_{1})},\|f_j\|_{L^{\infty}(Q_{1})}\right\}\leq \eta.
\]
Since \(\eta<\eta_1\), it follows that
\[
\max\left\{\|u_j\|_{L^\infty(Q_1)},[\mathfrak{a}_j]_{C^{0,1}_{x}(Q_{1})},\|f_j\|_{L^{\infty}(Q_{1})}\right\}\leq \eta_1.
\]
Thus, Lemma \ref{LocalLiplemaprobtransl} applies, yielding
\[
\|D  u_{j}\|_{L^{\infty}(Q_{3/4})}\leq 1.
\]
Using the lower bound for \(\ell_{j}\), we obtain
\begin{eqnarray}\label{limgradient}
|D u_{j}(x,t)+\vec{q}|\geq |\vec{q}|-|D u_{j}|\geq 1,\, (x,t)\in Q_{3/4}.    
\end{eqnarray}
Define \(v_{j}(x,t)=u_{j}(x,t)+\vec{q}\cdot x\). Then
\[
\partial_{t}v_{j}-\left(|Dv_j|^\mathfrak{p}+\mathfrak{a}_{j}(x,t)|Dv_j|^\mathfrak{q}\right)\Delta_{p}^{\rm{N}}v_{j}=f_{j}(x,t) \,\,\ \text{in}\,\,\, Q_{1},
\]
where $\vec{q},\,\mathfrak{a}_j$ and $f_j$ are given by \eqref{aux_eqHolder}. By \eqref{ineq1Teo5.4}, we obtain
\[
\|v_{j}\|_{L^{\infty}(Q_{1})}\leq \|u_{j}\|_{L^{\infty}(Q_{1})}+\left|\frac{\ell_{j}}{(1-\theta)^{j}}\right|\leq 2+2(1-\theta)^{-1}=\mathrm{C^{\ast}}.
\]
Consequently, by the local Lipschitz regularity in the spatial variable (Lemma \ref{Localholderlemaprobtransl}), it follows that
\[
\|D v_{j}\|_{L^{\infty}(Q_{3/4})}\leq \mathrm{C}\left(\|v_{j}\|_{L^{\infty}(Q_{1})}+\|v_{j}\|_{L^{\infty}(Q_{1})}^{\frac{1}{1+\mathfrak{p}}}+\|f_{j}\|_{L^{\infty}(Q_{1})}^{\frac{1}{1+\mathfrak{p}}}\right)\leq \tilde{\mathrm{C}}_{\ast},
\]
where \(\tilde{\mathrm{C}}_{\ast}>0\) depends only on \(n\), \(p\), \(\mathfrak{p}\), \(\mathfrak{q}\), \(\mathfrak{a}_{-}\), and \(\mathfrak{A}_{0}\). Note that \(v_{j}\) satisfies
\[
\partial_t v_{j}-\sum_{i,k=1}^{n}a_{ik}^{j}(x,t,D v_{j})\frac{\partial^{2}v_{j}}{\partial x_{i}\partial x_{k}}(x,t)=f_{j}(x,t)\,\,\, \text{in}\,\,\, Q_{3/4},
\]
where
\[
a^{j}_{ik}(x,t,\xi)=\mathscr{H}_{j}(x,t,\xi)\left(\delta_{ik}+(p-2)\frac{\xi_{i}\xi_{k}}{|\xi+\vec{q}|^{2}}\right).
\]
This operator is uniformly parabolic, with ellipticity constants depending only on \(n\), \(p\), \(\mathfrak{p}\), \(\mathfrak{q}\), \(\mathfrak{a}_{-}\), and \(\mathfrak{A}_{0}\), and is smooth with respect to the gradient variable (in view of \eqref{limgradient}). By \cite[Lemma 12.13]{Lieberman96}, we conclude that \(v_{j}\in C^{1+\tilde{\alpha},\frac{1+\tilde{\alpha}}{2}}(Q_{3/4})\) for some \(\tilde{\alpha}\in (0,1)\) depending only on the data. Moreover,
\[
\|D v_{j}\|_{C^{0,\tilde{\alpha}}(\Omega')}\leq \mathrm{C}(n, p,\mathfrak{p},\mathfrak{q},\mathfrak{a}_{-},\mathfrak{A}_{0},\operatorname{dist}(\Omega',\partial_{par}Q_{3/4})),\,\, \forall \Omega'\subset\subset Q_{3/4}.
\]
Proceeding as in \cite[Lemma 4.1]{AttouRuos20}, we conclude the proof.
\end{proof}


\subsection*{Acknowledgments}
FAPESP (Brazil) supported J.~da~Silva~Bessa under Grant No.~2023/18447-3. J.~V.~da~Silva received partial support from CNPq (Brazil) under Grant No.~307131/2022-0, from FAEPEX--UNICAMP (Grant No.~2441/23, Editais Especiais--PIND--Projetos Individuais, 03/2023), and from Chamada CNPq/MCTI No.~10/2023 (Faixa B---Consolidated Research Groups) under Grant No.~420014/2023-3. Additional support was provided by FAPESP (Brazil) under Grant No.~2025/09344-1 (Special Programs---Special Projects---First Projects, 2025, 1st Cycle). G.~S.~S\'{a} acknowledges support from the Centro de Modelamiento Matem\'{a}tico (CMM), BASAL Fund FB210005, a center of excellence funded by ANID (Chile).

\subsection*{Declarations}

\subsection*{Conflict of interest}

On behalf of all authors, the corresponding author states that there is no conflict of interest

\subsection*{Data availability statement}

Data availability does not apply to this article as no new data were created or analyzed in this study.

\end{document}